\documentclass[12pt]{article}%
\usepackage{amsfonts}
\usepackage{amsmath}
\usepackage{amssymb}
\usepackage{graphicx}%
\providecommand{\U}[1]{\protect\rule{.1in}{.1in}}
\newtheorem{theorem}{Theorem}

\newtheorem{corollary}[theorem]{Corollary}

\newtheorem{definition}[theorem]{Definition}
\newtheorem{example}[theorem]{Example}

\newtheorem{lemma}[theorem]{Lemma}

\newtheorem{proposition}[theorem]{Proposition}
\newtheorem{remark}[theorem]{Remark}

\newenvironment{proof}[1][Proof]{\noindent\textbf{#1.} }{\ \rule{0.5em}{0.5em}}
\begin{document}

\title{Exact bounds for efficient consistent matrices obtained from a reciprocal
matrix }
\author{Susana Furtado\thanks{Email: sbf@fep.up.pt. orcid.org/0000-0003-0395-5972. The
work of this author was supported by FCT- Funda\c{c}\~{a}o para a Ci\^{e}ncia
e Tecnologia, under project UID/04561/2025.} \thanks{Corresponding author.}\\CEMS.UL and Faculdade de Economia \\Universidade do Porto\\Rua Dr. Roberto Frias\\4200-464 Porto, Portugal
\and Charles R. Johnson \thanks{Email: crjmatrix@gmail.com. }\\225 West Tazewells Way\\Williamsburg, VA 23185}
\maketitle

\begin{abstract}
For a given reciprocal matrix $A$, we give a union of matrix intervals in
which any consistent matrix obtained from an efficient vector for $A$ lies,
and, conversely, any consistent matrix in this union comes from an efficient
vector for $A$. The maximal sets of entries in the lower and upper bound
matrices of each interval that are attainable by some consistent matrix in the
interval are described. This allows us to understand which subsets of the
alternatives lie above which other subsets in all efficient orders for each
interval. As a result, the partial order on the alternatives dictated by the
efficient vectors follows. Then, we use the tools developed to also show that,
when the $n$-by-$n$ reciprocal matrices $A,B$ are simple perturbed consistent
matrices, or $n=4$, the sets of efficient vectors for $A$ and $B$ coincide
only if $A=B.$

\end{abstract}



\section{Introduction}

An $n$-by-$n$ entry-wise positive matrix $A=[a_{ij}]$ is called
\emph{reciprocal} if $a_{ji}=\frac{1}{a_{ij}},$ for each pair $1\leq i,j\leq
n.$ In particular, all diagonal entries are $1$. The entries represent
pair-wise ratio comparisons among $n$ alternatives, so that such an $A$ is
also called a \emph{pair-wise comparison matrix}. We denote the set of all
such matrices by $\mathcal{PC}_{n}$. In a variety of multi-criterion decision
models, a cardinal ranking vector $w\in\mathbb{R}_{+}^{n}$ (the set of
positive $n$-vectors) is desired, for an $A\in\mathcal{PC}_{n}$, to represent
relative weights for the $n$ alternatives.

If $a_{ij}a_{jk}=a_{ik}$ for all triples $1\leq i,j,k\leq n$, $A\in
\mathcal{PC}_{n}$ is called \emph{consistent}. In this case, $A=ww^{(-T)}$, in
which $w\in\mathbb{R}_{+}^{n}$ and $w^{(-T)}$ is the entry-wise inverse of the
transpose of $w.$ Vector $w$, which is projectively unique (that is, unique up
to a factor of scale), is the natural ranking vector. However, in practice,
consistent matrices seldom arise. So, we may wish to approximate
$A\in\mathcal{PC}_{n}$ by a consistent matrix, in order to choose $w$. In an
early model \cite{saaty1977}, Saaty recommended the right Perron eigenvector
of $A\in\mathcal{PC}_{n}$ as the cardinal ranking vector. When $A$ is
consistent, this is the correct vector, but, otherwise, the Perron eigenvector
may not be the best choice \cite{blanq2006,Bozoki2014,FJ5}.

A ranking vector $w$ obtained from $A\in\mathcal{PC}_{n}$ should, at least, be
one upon which no global improvement is possible (Pareto optimality). This
means that $\left\vert A-vv^{(-T)}\right\vert \leq\left\vert A-ww^{(-T)}%
\right\vert $, entry-wise in absolute value, implies that $v\in\mathbb{R}%
_{+}^{n}$ is proportional to $w.$ In this event, $w$ is called
\emph{efficient} for $A\in\mathcal{PC}_{n}$ \cite{blanq2006, FJ1} and the
ordering of the alternatives indicated by the magnitudes of the entries of $w$
is called an \emph{efficient order}. We denote the set of all such vectors by
$\mathcal{E}(A)$. This set is (projectively) infinite$,$ unless $A$ is
consistent, in which case it consists of the positive multiples of any column
of $A$. A characterization of efficiency of a vector $w\in\mathbb{R}_{+}^{n}$
in terms of a directed graph appeared in \cite{blanq2006} (see also \cite{FJ2}).

There are some vectors that may be calculated from $A\in\mathcal{PC}_{n},$ via
a fixed formula, and are universally efficient. In \cite{blanq2006} it was
shown that the geometric mean of the columns of $A$ is efficient for $A.$
Later \cite{FJ2}, this result was extended to any weighted geometric mean of
the columns. More recently \cite{FJ7}, the geometric mean of the right and
entry-wise inverse left Perron vectors of $A$ was shown to be efficient for
$A$.

An inductive description of $\mathcal{E}(A)$ was presented in \cite{FJ4}. A
characterization of $\mathcal{E}(A)$, for inconsistent $A\in\mathcal{PC}_{n}$,
as a union of at most $\frac{(n-1)!}{2}$ convex sets was given in \cite{FJ3}.
Each of these sets is associated with a Hamiltonian cycle product $<1$ in the matrix.

In \cite{FJ8} we presented a matrix interval, depending on the entries of
$A\in\mathcal{PC}_{n}$, for the matrix $W=ww^{(-T)}$ obtained from a
$w\in\mathcal{E}(A)$. As a single matrix interval, this is the best possible,
since any entry of the lower and upper bound matrices is attainable for some
$w\in\mathcal{E}(A)$. This interval allowed us to characterize the matrices
$A\in\mathcal{PC}_{n}$ for which all efficient orders are the same. This is
especially important in applications in which an ordinal ranking is the goal.
However, the obtained matrix interval may include consistent matrices that do
not result from an efficient vector. We also have noticed that, for
$A,B\in\mathcal{PC}_{n}$, if $\mathcal{E}(A)=\mathcal{E}(B)$ then the given
matrix intervals coincide for $A$ and $B.$ If $n=3$, this implies that $A=B.$

Here we give a union of at most $\frac{(n-1)!}{2}$ matrix intervals in which
any matrix $W=ww^{(-T)}$ obtained from a $w\in\mathcal{E}(A)$ lies and such
that any consistent matrix in it comes from an efficient vector. This improves
the bounds in \cite{FJ8} and gives another way of describing $\mathcal{E}(A)$.
This new result can give more precise information than the one in \cite{FJ8}
concerning the possible efficient orders for $A$, or about how some
alternatives are positioned relative to the others$.$ We may see when some
alternatives lie above the remaining ones for all efficient vectors from a
particular interval, thus giving a complete description of the partial order
based on the pair-wise comparisons. It may happen that one alternative, or a
set of alternatives, is higher than the others in every efficient order. We
also show that $\mathcal{E}(A)=\mathcal{E}(B)$ implies $A=B$ when
$A,B\in\mathcal{PC}_{n}$ are simple perturbed consistent matrices, that is,
are obtained from a consistent matrix by modifying one pair of reciprocal
entries (this includes the matrices in $\mathcal{PC}_{3}$), and when
$A,B\in\mathcal{PC}_{4}.$

We start in Section \ref{s1} with some useful background. In Section
\ref{spath} we describe the so called path matrices associated with a
Hamiltonian cycle in $A\in\mathcal{PC}_{n}$ and give some of their properties.
These matrices are used in Section \ref{s4} to give a union of a finite number
of matrix intervals such that a consistent matrix lies in it if and only if it
is obtained from an efficient vector. In Section \ref{smax} we describe the
maximal sets of entries of the lower (upper) bound matrix of each interval
attained by a consistent matrix in that interval. Some consequences are given.
In Section \ref{sorder} we study the possible efficient orders corresponding
to the consistent matrices in each interval of our union of intervals. In
particular, we characterize when one set of alternatives is above another set
in all efficient orders. In the remaining sections we focus on the problem of
when $\mathcal{E}(A)=\mathcal{E}(B)$ implies $A=B$. In Section \ref{s60} we
give some general comments and identify some circumstances in which this
implication holds. Then, in Section \ref{s6}, we show this result when both
$A$ and $B$ are simple perturbed consistent matrices. In Section \ref{snec} we
give necessary conditions for $\mathcal{E}(A)=\mathcal{E}(B)$ that allow us to
obtain a condition under which $\mathcal{E}(A)=\mathcal{E}(B)$ implies $A=B$.
The given necessary conditions are also used in Section \ref{sn4} to show that
the claimed implication holds when $n=4.$ We conclude with some observations
in Section \ref{scon}.

\section{Preliminaries\label{s1}}

We start with some additional notation that we will use. Let $n>1$ be an
integer. By $M_{n}$ we denote the set of $n$-by-$n$ real matrices. A product
of a permutation and a positive diagonal matrix is called a\ \emph{monomial
matrix}. We denote $N=\{1,\ldots,n\}$ and $\mathcal{N}=\left\{  (x,y)\in
N\times N:x\neq y\right\}  .$ Given a set $S\subseteq N\times N$ and $A\in
M_{n}$, we denote by $A[S]$ the entries of $A$ indexed by $S.$ We write
$A[S]=B[S]$ if, in each position indexed by $S,$ the entries in $A$ and $B\in
M_{n}$ coincide. We denote by $A^{(-T)}$ the entry-wise inverse of the
transpose of $A$. Note that $A\in\mathcal{PC}_{n}$ if and only if $A$ is
positive and $A=A^{(-T)}.$ For $A\in M_{n}$, $a_{ij}$ denotes the $i,j$ entry
of $A$. Given a vector $w\in\mathbb{R}_{+}^{n},$ we denote its $i$th entry by
$w_{i}$. By $W$ we denote the consistent matrix $ww^{(-T)}=[\frac{w_{i}}%
{w_{j}}]$.

Any matrix in $\mathcal{PC}_{2}$ is consistent. We denote by $\mathcal{PC}%
_{n}^{0}$, $n\geq3$, the set of matrices in $\mathcal{PC}_{n}$ that are not
consistent. Each of $\mathcal{PC}_{n}$, $\mathcal{PC}_{n}^{0}$ and the set of
consistent matrices is closed under monomial similarity. Such similarities
interface well with $\mathcal{E}(A)$.

\begin{lemma}
\label{lsim}\cite{FJ1} Suppose that $A\in\mathcal{PC}_{n}.$ If $S\in M_{n}$ is
a monomial matrix, then $\mathcal{E}(SAS^{-1})=S\mathcal{E}(A)$.
\end{lemma}

\bigskip

We call a sequence $\tau:\tau_{1}\tau_{2}\cdots\tau_{n}\tau_{1}$, in which
$\tau_{1}\tau_{2}\cdots\tau_{n}$ is a permutation of $N$, a \emph{Hamiltonian
cycle} (H-cycle) in $N$. Note that an H-cycle can be listed starting with any
index. For $A\in\mathcal{PC}_{n}$, we denote by $\tau(A)$ the product of the
entries of $A$ along $\tau.$ We denote by $\Gamma(A)$ the set of H-cycles
$\tau$ in $N$ such that $\tau(A)<1.$ If $A$ is inconsistent, there is at least
one H-cycle in $\Gamma(A)$. Also, there are at most $\frac{(n-1)!}{2}$ such cycles.

If $A\in\mathcal{PC}_{n}$ and $\tau$ is an H-cycle in $N$, we denote by
$P_{A,\tau}(i,j)$, $i\neq j$, the product of the entries in $A$ along the path
from $i$ to $j$ in $\tau.$ We have $P_{A,\tau}(\tau_{1},\tau_{i+1})=P_{A,\tau
}(\tau_{1},\tau_{i})P_{A,\tau}(\tau_{i},\tau_{i+1})=P_{A,\tau}(\tau_{1}%
,\tau_{i})a_{\tau_{i},\tau_{i+1}},$ $i=\{2,\ldots,n-1\},$ and $\tau
(A)=P_{A,\tau}(\tau_{1},\tau_{n})P_{A,\tau}(\tau_{n},\tau_{1}).$

\begin{example}
\bigskip If $A=[a_{ij}]\in\mathcal{PC}_{4}$ and $\tau=12341$, then
$\tau(A)=a_{12}a_{23}a_{34}a_{41}$. Also, $P_{A,\tau}(1,3)=a_{12}a_{23}$,
$P_{A,\tau}(3,2)=a_{34}a_{41}a_{12}$ and $P_{A,\tau}(1,4)=P_{A,\tau
}(1,3)P_{A,\tau}(3,4)=a_{12}a_{23}a_{34}.$
\end{example}

From now on, we focus on matrices in $A\in\mathcal{PC}_{n}^{0}$.\bigskip

Next we describe $\mathcal{E}(A)$ as a finite union of convex sets \cite{FJ3}.
For $\tau\in\Gamma(A),$ let%
\[
\mathcal{E}_{\tau}(A)=\left\{  w\in\mathbb{R}_{+}^{n}:w_{\tau_{1}}\geq
P_{A,\tau}(\tau_{1},\tau_{2})w_{\tau_{2}}\geq\cdots\geq P_{A,\tau}(\tau
_{1},\tau_{n})w_{\tau_{n}}\geq\tau(A)w_{\tau_{1}}\right\}  .
\]

The set $\mathcal{E}_{\tau}(A)$ is convex and is the cone generated by the
(projectively distinct) $n$ vectors satisfying
\begin{equation}
w_{\tau_{1}}\geq P_{A,\tau}(\tau_{1},\tau_{2})w_{\tau_{2}}\geq\cdots\geq
P_{A,\tau}(\tau_{1},\tau_{n})w_{\tau_{n}}\geq\tau(A)w_{\tau_{1}},
\label{ineqe}%
\end{equation}
with $n-1$ inequalities replaced by equalities. The $n$ rays corresponding to
these vectors are the extremes of the cone $\mathcal{E}_{\tau}(A)$. Since
$\tau(A)\neq1$, there is no vector $w$ satisfying (\ref{ineqe}) with all
inequalities replaced by equalities.

\begin{theorem}
\label{TeEaconvexunion}\cite{FJ3} Let $A\in\mathcal{PC}_{n}^{0}$. Then%
\[
\mathcal{E}(A)=\bigcup_{\tau\in\Gamma(A)}\mathcal{E}_{\tau}(A).
\]

\end{theorem}

Since $\mathcal{E}(A)$ is known to be connected \cite{blanq2006,FJ4}, for any
$\tau\in\Gamma(A),$ there is another $\mu\in\Gamma(A)$ such that
$\mathcal{E}_{\tau}(A)$ and $\mathcal{E}_{\mu}(A)$ intersect. However,
$\mathcal{E}(A)$ is not generally convex, but it is, for example, when
$\#\Gamma(A)=1$ (in which case $n\leq4$), or if $A$ is obtained from a
consistent matrix by modifying one pair of reciprocal entries
\cite{AbeleBozoki2016,CFF} (See Section \ref{s6}).

\bigskip

An important fact noticed in \cite{FJ8} is the following.

\begin{theorem}
\label{Rlow}Let $A\in\mathcal{PC}_{n}^{0}$. If $\tau\in\Gamma(A)$ and
$w\in\mathcal{E}_{\tau}(A),$ then $\frac{w_{i}}{w_{j}}\geq P_{A,\tau}(i,j)$
for any $i,j\in N$, $i\neq j.$ Moreover, there is a $w\in\mathcal{E}_{\tau
}(A)$ such that $\frac{w_{i}}{w_{j}}=P_{A,\tau}(i,j).$\bigskip
\end{theorem}

\bigskip

In \cite{FJ8}, the best single matrix interval for the consistent matrices
$ww^{(-T)}=\left[  \frac{w_{i}}{w_{j}}\right]  $ obtained from an efficient
vector $w$ for $A\in\mathcal{PC}_{n}^{0}$ has been given. Let $L_{A}%
=[l_{ij}]\in M_{n}$ be defined by $l_{ii}=1$ and
\[
l_{ij}=\min_{\tau\in\Gamma(A)}P_{A,\tau}(i,j),
\]
$i,j\in N,$ $i\neq j.$ Let $U_{A}=(L_{A})^{(-T)}$ be the entry-wise inverse of
$L_{A}^{T}.$

\begin{theorem}
\label{tminmax}\cite{FJ8} Let $A\in\mathcal{PC}_{n}^{0}$. Let $w\in
\mathcal{E}(A)$ and $W=ww^{(-T)}.$ Then%
\[
L_{A}\leq W\leq U_{A}.
\]

\end{theorem}

In \cite{FJ6} the existence of a unique efficient order for $A\in
\mathcal{PC}_{n}^{0}$ was first studied. Then, in \cite{FJ8}, this was
characterized using the matrices $L_{A}$ and $U_{A}$.

\begin{theorem}
\label{Tuo}\cite{FJ8} Let $A\in\mathcal{PC}_{n}^{0}$ and $L_{A}=[l_{ij}]$. The
following are equivalent:

\begin{enumerate}
\item there is only one efficient order for $A$;

\item there is a permutation $i_{1}i_{2}\cdots i_{n}$ of $N$ such that
$l_{i_{t},i_{t+1}}\geq1$ for $t=1,,\ldots,n-1$; and

\item $L_{A}$ has exactly $\frac{n^{2}-n}{2}$ off-diagonal entries $\geq1$.
\end{enumerate}
\end{theorem}

\section{Path matrices\label{spath}}

Given $A\in\mathcal{PC}_{n}^{0}$ and $\tau$ an H-cycle in $N,$ we call the
matrix $P_{A,\tau}=[p_{ij}]$ defined by $p_{ij}=P_{A,\tau}(i,j)$ and
$p_{ii}=1,$ $i,j\in N,$ $i\neq j$, $\ $the $\tau$ \emph{path matrix} for $A.$
The matrix $P_{A,\tau}$ is determined by the H-cycle $\tau$ and the entries in
$A$ along it. In the next section we will use the path matrices for $A$ and
its H-cycles in $\Gamma(A)$ to give an exact collection of intervals in which
the consistent matrices constructed from efficient vectors for $A$ lie.

Recall that
\begin{equation}
p_{ij}p_{ji}=\tau(A), \label{proij}%
\end{equation}
for $i\neq j$, and observe that
\[
P_{A,\tau}=(P_{A,\operatorname*{rev}\tau})^{(-T)},
\]
in which $\operatorname*{rev}\tau$ denotes the H-cycle that is reverse to
$\tau.$

\begin{remark}
\label{tequalprod}Let $A,B\in\mathcal{PC}_{n}^{0}$. Let $\tau$ and $\nu$ be
H-cycles in $N$. If there are $i,j\in N$, $i\neq j,$ such that $P_{B,\nu
}(i,j)\leq P_{A,\tau}(i,j)$ and $P_{B,\nu}(j,i)\leq P_{A,\tau}(j,i)$, then
$\nu(B)\leq\tau(A).$
\end{remark}

\begin{example}
\label{ex4by4}Let
\[
A=\left[
\begin{array}
[c]{cccc}%
1 & 1 & a_{13} & a_{14}\\
1 & 1 & 1 & a_{24}\\
\frac{1}{a_{13}} & 1 & 1 & 1\\
\frac{1}{a_{14}} & \frac{1}{a_{24}} & 1 & 1
\end{array}
\right]  \in\mathcal{PC}_{4}^{0}.
\]
There are $6$ H-cycles in $\{1,2,3,4\}:$%
\begin{equation}
\alpha=12341;\quad\beta=12431;\quad\gamma=13241; \label{abg}%
\end{equation}
and their reverses
\begin{equation}
\alpha^{\prime}=14321;\quad\beta^{\prime}=13421;\quad\gamma^{\prime}=14231.
\label{abg1}%
\end{equation}
We have
\begin{align*}
\alpha(A)  &  =\frac{1}{a_{14}},\quad\beta(A)=\frac{a_{24}}{a_{13}}%
,\quad\gamma(A)=\frac{a_{13}a_{24}}{a_{14}},\\
\alpha^{\prime}(A)  &  =a_{14},\quad\beta^{\prime}(A)=\frac{a_{13}}{a_{24}%
},\quad\gamma^{\prime}(A)=\frac{a_{14}}{a_{13}a_{24}}.
\end{align*}
Also,
\[
P_{A,\alpha}=\left[
\begin{array}
[c]{cccc}%
1 & 1 & 1 & 1\\
\frac{1}{a_{14}} & 1 & 1 & 1\\
\frac{1}{a_{14}} & \frac{1}{a_{14}} & 1 & 1\\
\frac{1}{a_{14}} & \frac{1}{a_{14}} & \frac{1}{a_{14}} & 1
\end{array}
\right]  ,\quad P_{A,\alpha^{\prime}}=\left[
\begin{array}
[c]{cccc}%
1 & a_{14} & a_{14} & a_{14}\\
1 & 1 & a_{14} & a_{14}\\
1 & 1 & 1 & a_{14}\\
1 & 1 & 1 & 1
\end{array}
\right]  ,
\]%
\[
P_{A,\beta}=\left[
\begin{array}
[c]{cccc}%
1 & 1 & a_{24} & a_{24}\\
\beta(A) & 1 & a_{24} & a_{24}\\
\frac{1}{a_{13}} & \frac{1}{a_{13}} & 1 & \beta(A)\\
\frac{1}{a_{13}} & \frac{1}{a_{13}} & 1 & 1
\end{array}
\right]  ,\text{ }\quad P_{A,\beta^{\prime}}=\left[
\begin{array}
[c]{cccc}%
1 & \beta^{\prime}(A) & a_{13} & a_{13}\\
1 & 1 & a_{13} & a_{13}\\
\frac{1}{a_{24}} & \frac{1}{a_{24}} & 1 & 1\\
\frac{1}{a_{24}} & \frac{1}{a_{24}} & \beta^{\prime}(A) & 1
\end{array}
\right]  ,\text{ }%
\]

\end{example}

\[
P_{A,\gamma}=\left[
\begin{array}
[c]{cccc}%
1 & a_{13} & a_{13} & a_{13}a_{24}\\
\frac{a_{24}}{a_{14}} & 1 & \gamma(A) & a_{24}\\
\frac{a_{24}}{a_{14}} & 1 & 1 & a_{24}\\
\frac{1}{a_{14}} & \frac{a_{13}}{a_{14}} & \frac{a_{13}}{a_{14}} & 1
\end{array}
\right]  ,\text{ }\quad P_{A,\gamma^{\prime}}=\left[
\begin{array}
[c]{cccc}%
1 & \frac{a_{14}}{a_{24}} & \frac{a_{14}}{a_{24}} & a_{14}\\
\frac{1}{a_{13}} & 1 & 1 & \frac{a_{14}}{a_{13}}\\
\frac{1}{a_{13}} & \gamma^{\prime}(A) & 1 & \frac{a_{14}}{a_{13}}\\
\frac{1}{a_{13}a_{24}} & \frac{1}{a_{24}} & \frac{1}{a_{24}} & 1
\end{array}
\right]  .
\]

We next see when two path matrices associated with H-cycles with cycle
products $<1$ (possibly, for two different matrices) are equal. We first give
an important lemma that shows how $P_{A,\tau}$ changes under a monomial
similarity on $A.$ It allows us to consider the matrix $A$ in a special form.

\begin{lemma}
\label{LpermutLA}Let $A\in\mathcal{PC}_{n},$ let $D\in M_{n}$ be a positive
diagonal matrix and $Q=[q_{ij}]\in M_{n}$ be a permutation matrix. Let $\tau$
be an H-cycle in $N.$ Then, $D^{-1}P_{A,\tau}D=P_{D^{-1}AD,\tau}$, and
$Q^{T}P_{A,\tau}Q=P_{Q^{T}AQ,\rho}$, in which $\rho$ is the H-cycle satisfying
$q_{\tau_{i},\rho_{i}}=1.$
\end{lemma}

\begin{proof}
Let $D=\operatorname*{diag}(d_{1},\ldots,d_{n})$ and $i,j\in N,$ $i\neq j.$
Let $\tau=\tau_{1}\ldots\tau_{n}\tau_{1}$ with $\tau_{1}=i$ and $\tau_{k}=j$.
We have
\begin{align*}
P_{D^{-1}AD,\tau}(i,j)  &  =\frac{d_{\tau_{2}}}{d_{\tau_{1}}}a_{\tau_{1}%
\tau_{2}}\frac{d_{\tau_{3}}}{d_{\tau_{2}}}a_{\tau_{2}\tau_{3}}\cdots
\frac{d_{\tau_{k}}}{d_{\tau_{k-1}}}a_{\tau_{k-1}\tau_{k}}\\
&  =\frac{d_{\tau_{k}}}{d_{\tau_{1}}}a_{\tau_{1}\tau_{2}}a_{\tau_{2}\tau_{3}%
}\cdots a_{\tau_{k-1}\tau_{k}}=\frac{d_{j}}{d_{i}}P_{A,\tau}(i,j).
\end{align*}
By definition of $\rho,$ for $B\in M_{n},$ the $\rho_{i},\rho_{j}$ entry of
$Q^{T}BQ$ is the $\tau_{i},\tau_{j}$ entry of $B$ (that is, the sequence of
entries in $B$ and in $Q^{T}BQ$ along $\tau$ and $\rho,$ respectively,
coincide). Thus, the $\rho_{i},\rho_{j}$ entry of $Q^{T}P_{A,\tau}Q$ is
$P_{A,\tau}(\tau_{i},\tau_{j}).$ If $i<j,$ the $\rho_{i},\rho_{j}$ entry of
$P_{Q^{T}AQ,\rho}$ is $(Q^{T}AQ)_{\rho_{i},\rho_{i+1}}\cdots(Q^{T}%
AQ)_{\rho_{j-1},\rho_{j}}=a_{\tau_{i},\tau_{i+1}}\cdots a_{\tau_{j-1},\tau
_{j}}=P_{A,\tau}(\tau_{i},\tau_{j}).$ The proof is similar if $i>j$.
\end{proof}

\bigskip

As in the proof of Lemma \ref{LpermutLA}, it follows that, if $D\in M_{n}$ is
a positive diagonal matrix and $\tau$ is an H-cycle in $N,$ then $\tau
(A)=\tau(D^{-1}AD).$

\begin{theorem}
\label{thequalpath}Let $A,B\in\mathcal{PC}_{n}^{0}$. If $\tau\in\Gamma(A)$,
$\nu\in\Gamma(B)$ and $P_{A,\tau}=P_{B,\nu}$, then $\tau=\nu.$ Moreover, the
sequence of entries in $A$ and in $B$ along $\tau$ ($=\nu$) coincide. In
particular, $\tau(A)=\nu(A)$.
\end{theorem}

\begin{proof}
Taking into account Lemma \ref{LpermutLA}, by a simultaneous monomial
similarity on $A$ and $B,$ we assume that $\tau=123\cdots n1$ and
$a_{i,i+1}=1$ for $i=1,\ldots,n-1.$ Then $a_{n1}=\tau(A).$ Let $\nu=\nu_{1}%
\nu_{2}\cdots\nu_{n}\nu_{1},$ with $\nu_{1}=1.$ Suppose that $P_{A,\tau
}=P_{B,\nu}$. In order to get a contradiction, suppose that $\nu\neq\tau,$
that is, there is $i\in\{2,\ldots,n-1\}$ such that $\nu_{i}>\nu_{i+1}.$ For
such an $i$, we have $P_{A,\tau}(1,\nu_{i})=P_{A,\tau}(1,\nu_{i+1})=1.$ Also,
$P_{A,\tau}(\nu_{i},\nu_{i+1})=\tau(A).$ Thus, if $P_{A,\tau}=P_{B,\nu},$ we
have
\begin{align*}
1  &  =P_{A,\tau}(1,\nu_{i+1})=P_{B,\nu}(1,\nu_{i+1})=P_{B,\nu}(1,\nu
_{i})P_{B,\nu}(\nu_{i},\nu_{i+1})\\
&  =P_{A,\tau}(1,\nu_{i})P_{A,\tau}(\nu_{i},\nu_{i+1})=\tau(A),
\end{align*}
a contradiction, since $\tau(A)<1$. Thus, $\tau=\nu.$ As, for $i\in
\{1,\ldots,n-1\},$ we have $a_{i,i+1}=P_{A,\tau}(i,i+1)=P_{B,\tau
}(i,i+1)=b_{i,i+1}$ and $a_{n,1}=P_{A,\tau}(n,1)=P_{B,\tau}(n,1)=b_{n,1},$ the
second claim also follows.
\end{proof}

\bigskip

Taking Theorem \ref{Rlow} into account, if $\mathcal{E}_{\tau}(A)=\mathcal{E}%
_{\nu}(B)$, then $P_{A,\tau}(i,j)=P_{B,\nu}(i,j)$, for any $i,j\in N$, that
is, $P_{A,\tau}=P_{B,\nu}$. Thus, from Theorem \ref{thequalpath}, we get the following.

\begin{corollary}
\label{cPeq}Let $A,B\in\mathcal{PC}_{n}^{0}$. Let $\tau\in\Gamma(A)$ and
$\nu\in\Gamma(B)$. Then, $\mathcal{E}_{\tau}(A)=\mathcal{E}_{\nu}(B)\ $if and
only if $\tau=\nu$ and the sequence of entries in $A$ and in $B$ along $\tau$
($=\nu$) coincide.
\end{corollary}

\section{Exact Bounds for $W$\label{s4}}

Now, we may give our finite union of matrix intervals in which all consistent
matrices are exactly those that come from vectors in $\mathcal{E}(A).$

\begin{theorem}
\label{tt01}Let $A\in\mathcal{PC}_{n}^{0}$. Let $\tau\in\Gamma(A)$,
$w\in\mathbb{R}_{+}^{n}$ and $W=\left[  \frac{w_{i}}{w_{j}}\right]  .$ Then,
$w\in\mathcal{E}_{\tau}(A)$ if and only if
\begin{equation}
P_{A,\tau}\leq W\leq P_{A,\tau}^{(-T)}. \label{intt01}%
\end{equation}

\end{theorem}

\begin{proof}
Suppose that $w\in\mathcal{E}_{\tau}(A)$. By Theorem \ref{Rlow}, $\frac{w_{i}%
}{w_{j}}\geq P_{A,\tau}(i,j)$ for all $i,j,$ implying the lower bound. The
upper bound follows from the lower bound taking into account that
$W=W^{(-T)}.$

Now suppose that (\ref{intt01}) holds. Then, $\frac{w_{\tau_{i}}}%
{w_{\tau_{i+1}}}\geq P_{A,\tau}(\tau_{i},\tau_{i+1}),$ $i=1,\ldots,n-1,$ and
$\frac{w_{\tau_{n}}}{w_{\tau_{1}}}\geq P_{A,\tau}(\tau_{n},\tau_{1}).$ Thus,
$w\in\mathcal{E}_{\tau}(A).$
\end{proof}

\bigskip Since $P_{A,\tau}\leq W$ is equivalent to $W\leq P_{A,\tau}^{(-T)}$,
condition (\ref{intt01}) is equivalent to $P_{A,\tau}\leq W$. Also, observe
that, because of (\ref{proij}), if $p_{ij}$, $i\neq j$, is the $i,j$ entry of
$P_{A,\tau}$, then the $i,j$ entry of $P_{A,\tau}^{(-T)}$ is $\frac{1}%
{\tau(A)}p_{ij}$.

\begin{corollary}
\label{cpless}Let $A\in\mathcal{PC}_{n}^{0}$ and $\tau,\nu\in\Gamma(A).$ If
$P_{A,\tau}\leq P_{A,\nu}$ then $\mathcal{E}_{\nu}(A)\subseteq\mathcal{E}%
_{\tau}(A).$
\end{corollary}

\begin{proof}
Let $w\in\mathcal{E}_{\nu}(A).$ By Theorem \ref{tt01}, $P_{A,\nu}\leq W.$
Then, by the hypothesis, $P_{A,\tau}\leq W,$ implying $W\leq P_{A,\tau}%
^{(-T)}.$ By Theorem \ref{tt01}, $w\in\mathcal{E}_{\tau}(A).$
\end{proof}

\bigskip

As an immediate consequence of Theorems \ref{TeEaconvexunion} and \ref{tt01},
we give the main result of this section.

\begin{theorem}
\label{tt012}Let $A\in\mathcal{PC}_{n}^{0},$ $w\in\mathbb{R}_{+}^{n}$ and
$W=ww^{(-T)}$. Then, $w\in\mathcal{E}(A)$ if and only if there is $\tau
\in\Gamma(A)$ such that
\[
P_{A,\tau}\leq W\leq P_{A,\tau}^{(-T)}.
\]

\end{theorem}

Notice that the matrix $L_{A}$ in Theorem \ref{tminmax} is the entry-wise
minimum of the matrices $P_{A,\tau}$, $\tau\in\Gamma(A)$.

\begin{example}
\label{exbound}Let
\[
A=\left[
\begin{array}
[c]{cccc}%
1 & 1 & 3 & 7\\
1 & 1 & 1 & 2\\
\frac{1}{3} & 1 & 1 & 1\\
\frac{1}{7} & \frac{1}{2} & 1 & 1
\end{array}
\right]  .
\]
We have $\Gamma(A)=\{\alpha,\beta,\gamma\}$, with $\alpha,$ $\beta,$ $\gamma$
as in (\ref{abg})$.$ Moreover, $\alpha(A)=\frac{1}{7},$ $\beta(A)=\frac{2}{3}$
and $\gamma(A)=\frac{6}{7}.$ We have (see Example \ref{ex4by4})%
\begin{align*}
P_{A,\alpha}  &  =\left[
\begin{array}
[c]{cccc}%
1 & 1 & 1 & 1\\
\frac{1}{7} & 1 & 1 & 1\\
\frac{1}{7} & \frac{1}{7} & 1 & 1\\
\frac{1}{7} & \frac{1}{7} & \frac{1}{7} & 1
\end{array}
\right]  ,\quad P_{A,\beta}=\left[
\begin{array}
[c]{cccc}%
1 & 1 & 2 & 2\\
\frac{2}{3} & 1 & 2 & 2\\
\frac{1}{3} & \frac{1}{3} & 1 & \frac{2}{3}\\
\frac{1}{3} & \frac{1}{3} & 1 & 1
\end{array}
\right]  ,\quad P_{A,\gamma}=\left[
\begin{array}
[c]{cccc}%
1 & 3 & 3 & 6\\
\frac{2}{7} & 1 & \frac{6}{7} & 2\\
\frac{2}{7} & 1 & 1 & 2\\
\frac{1}{7} & \frac{3}{7} & \frac{3}{7} & 1
\end{array}
\right]  .\\
&
\end{align*}
If $w\in\mathcal{E}(A),$ one of the following holds:
\begin{equation}
\left[
\begin{array}
[c]{cccc}%
1 & 1 & 1 & 1\\
\frac{1}{7} & 1 & 1 & 1\\
\frac{1}{7} & \frac{1}{7} & 1 & 1\\
\frac{1}{7} & \frac{1}{7} & \frac{1}{7} & 1
\end{array}
\right]  \leq W\leq\left[
\begin{array}
[c]{cccc}%
1 & 7 & 7 & 7\\
1 & 1 & 7 & 7\\
1 & 1 & 1 & 7\\
1 & 1 & 1 & 1
\end{array}
\right]  ; \label{b1}%
\end{equation}%
\begin{equation}
\left[
\begin{array}
[c]{cccc}%
1 & 1 & 2 & 2\\
\frac{2}{3} & 1 & 2 & 2\\
\frac{1}{3} & \frac{1}{3} & 1 & \frac{2}{3}\\
\frac{1}{3} & \frac{1}{3} & 1 & 1
\end{array}
\right]  \leq W\leq\left[
\begin{array}
[c]{cccc}%
1 & \frac{3}{2} & 3 & 3\\
1 & 1 & 3 & 3\\
\frac{1}{2} & \frac{1}{2} & 1 & 1\\
\frac{1}{2} & \frac{1}{2} & \frac{3}{2} & 1
\end{array}
\right]  ; \label{b2}%
\end{equation}
or%
\begin{equation}
\left[
\begin{array}
[c]{cccc}%
1 & 3 & 3 & 6\\
\frac{2}{7} & 1 & \frac{6}{7} & 2\\
\frac{2}{7} & 1 & 1 & 2\\
\frac{1}{7} & \frac{3}{7} & \frac{3}{7} & 1
\end{array}
\right]  \leq W\leq\left[
\begin{array}
[c]{cccc}%
1 & \frac{7}{2} & \frac{7}{2} & 7\\
\frac{1}{3} & 1 & 1 & \frac{7}{3}\\
\frac{1}{3} & \frac{7}{6} & 1 & \frac{7}{3}\\
\frac{1}{6} & \frac{1}{2} & \frac{1}{2} & 1
\end{array}
\right]  . \label{b3}%
\end{equation}

\end{example}

\section{Maximal sets of attainable bounds\label{smax}}

For any $A\in\mathcal{PC}_{n}^{0}$, $\tau\in\Gamma(A)$ and $i,j\in N$, there
is a $w\in\mathcal{E}_{\tau}(A)$ such that $\frac{w_{i}}{w_{j}}=P_{A,\tau
}(i,j)$ (Theorem \ref{Rlow})$.$ Here we describe the maximal sets of positions
in which $W=\left[  \frac{w_{i}}{w_{j}}\right]  $ and $P_{A,\tau}$ can coincide.

\begin{definition}
Let $\tau=\tau_{1}\tau_{2}\cdots\tau_{n}\tau_{1},$ with $\tau_{1}=k,$ be an
H-cycle in $N$. For $i=1,\ldots,n-1,$ let $S^{(\tau_{i})}$ be the set of pairs
$(\tau_{i},\tau_{j}),$ with $j=i+1,\ldots,n.$ Let
\[
S_{k}^{(\tau)}=S^{(\tau_{1})}\cup S^{(\tau_{2})}\cup\cdots\cup S^{(\tau
_{n-1})}.
\]

\end{definition}

\begin{remark}
\label{rrow}Let $\tau=\tau_{1}\tau_{2}\cdots\tau_{n}\tau_{1},$ with $\tau
_{1}=k,$ be an H-cycle in $N$. We have that $(p,q)\in S_{k}^{(\tau)}$ if and
only if there are $i,j\in N,$ $i<j,$ such that $p=\tau_{i}$ and $q=\tau_{j}.$
In particular, $(p,q)\in S_{k}^{(\tau)}$ for all $q\in N,$ $q\neq p,$ if and
only if $p=k,$ and $(p,q)\in S_{k}^{(\tau)}$ for all $p\in N,$ $p\neq q,$ if
and only if $q=\tau_{n}.$ Moreover, if $(p,q),$ $(q,r)\in S_{k}^{(\tau)}$ then
$(p,r)\in S_{k}^{(\tau)}$. The set $S_{k}^{(\tau)}$ has $\frac{n^{2}-n}{2}$
pairs, none of which is the reversal of another. In fact, for any $p,q\in N,$
$p\neq q$, exactly one of $(p,q)$ or $(q,p)$ is in $S_{k}^{(\tau)}$.
\end{remark}

\begin{example}
\label{exsets}Let $n=4,$ and $\alpha$, $\beta,$ $\gamma$ be as (\ref{abg}). We
have
\begin{align*}
S_{1}^{(\alpha)}  &  =\{(1,2),(1,3),(1,4),(2,3),(2,4),(3,4)\},\\
S_{2}^{(\alpha)}  &  =\{(2,3),(2,4),(2,1),(3,4),(3,1),(4,1)\},\\
S_{3}^{(\alpha)}  &  =\{(3,4),(3,1),(3,2),(4,1),(4,2),(1,2)\},\\
S_{4}^{(\alpha)}  &  =\{(4,1),(4,2),(4,3),(1,2),(1,3),(2,3)\}.
\end{align*}%
\begin{align*}
S_{1}^{(\beta)}  &  =\{(1,2),(1,4),(1,3),(2,4),(2,3),(4,3)\},\\
S_{2}^{(\beta)}  &  =\{(2,4),(2,3),(2,1),(4,3),(4,1),(3,1)\},\\
S_{3}^{(\beta)}  &  =\{(3,1),(3,2),(3,4),(1,2),(1,4),(2,4)\},\\
S_{4}^{(\beta)}  &  =\{(4,3),(4,1),(4,2),(3,1),(3,2),(1,2)\}.
\end{align*}%
\begin{align*}
S_{1}^{(\gamma)}  &  =\{(1,3),(1,2),(1,4),(3,2),(3,4),(2,4)\},\\
S_{2}^{(\gamma)}  &  =\{(2,4),(2,1),(2,3),(4,1),(4,3),(1,3)\},\\
S_{3}^{(\gamma)}  &  =\{(3,2),(3,4),(3,1),(2,4),(2,1),(4,1)\},\\
S_{4}^{(\gamma)}  &  =\{(4,1),(4,3),(4,2),(1,3),(1,2),(3,2)\}.
\end{align*}
The sets $S_{k}^{(\operatorname*{rev}\alpha)},$ $S_{k}^{(\operatorname*{rev}%
\beta)},$ $S_{k}^{(\operatorname*{rev}\gamma)},$ $k=1,2,3,4,$ are obtained
from the sets $S_{k}^{(\alpha)},$ $S_{k}^{(\beta)},$ $S_{k}^{(\gamma)}$ by
reversing the pairs $(i,j)$ with $i\neq k.$
\end{example}

\bigskip

Each set $S_{k}^{(\tau)}$ uniquely determines $\tau$ and $k.$

\begin{proposition}
\label{rsdist}Let $k_{1},k_{2}\in N$ and $\tau,\nu$ be H-cycles in $N.$ If
$k_{1}\neq k_{2}$ or $\tau\neq\nu,$ then $S_{k_{1}}^{(\tau)}\neq S_{k_{2}%
}^{(\nu)}.$
\end{proposition}

\begin{proof}
The fact that $k_{1}\neq k_{2}$ implies $S_{k_{1}}^{(\tau)}\neq S_{k_{2}%
}^{(\nu)}$ follows from Remark \ref{rrow}. If $k\in N,$ $\tau=\tau_{1}\tau
_{2}\cdots\tau_{n}\tau_{1}$ and $\nu=\nu_{1}\nu_{2}\cdots\nu_{n}\nu_{1},$ with
$\tau_{1}=\nu_{1}=k,$ and $\tau\neq\nu,$ then there are $i,j\in N$ such that
$i$ precedes $j$ in $\tau_{1}\tau_{2}\cdots\tau_{n},$ and $j$ precedes $i$ in
$\nu_{1}\nu_{2}\cdots\nu_{n}.$ Then $(i,j)\in S_{k}^{(\tau)}$ but $(i,j)\notin
S_{k}^{(\nu)}$ (in fact $(j,i)\in S_{k}^{(\nu)}$).
\end{proof}

\bigskip

Each set $S_{k}^{(\tau)}$ determines a set of entries in the matrix
$P_{A,\tau}$. It contains the off-diagonal entries in the $k$th row and in the
$\tau_{n}$th column of $P_{A,\tau}$ (assuming $\tau_{1}=k$)$.$ Also note that
the set of entries determined by $S_{k}^{(\tau)}$ is the transpose of the one
determined by $S_{\tau_{n}}^{(\operatorname*{rev}\tau)}$. We include examples
when $n=4$ that will be helpful later.

\begin{example}
\label{exSa}Let $A$ be the matrix in Example \ref{ex4by4} and $\alpha
,\beta,\gamma$ and $\alpha^{\prime},\beta^{\prime},\gamma^{\prime}$ be the
H-cycles given in (\ref{abg}) and (\ref{abg1})$.$ The entries in $P_{A,\alpha
}$ indexed by the sets $S_{k}^{(\alpha)}$ (Example \ref{exsets}) are
highlighted next:
\begin{align*}
S_{1}^{(\alpha)}  &  :\left[
\begin{array}
[c]{cccc}%
1 &
\begin{tabular}
[c]{|c|}\hline
$1$\\\hline
\end{tabular}
&
\begin{tabular}
[c]{|c|}\hline
$1$\\\hline
\end{tabular}
& \smallskip%
\begin{tabular}
[c]{|c|}\hline
$1$\\\hline
\end{tabular}
\\
\frac{1}{a_{14}} & 1 &
\begin{tabular}
[c]{|c|}\hline
$1$\\\hline
\end{tabular}
& \smallskip%
\begin{tabular}
[c]{|c|}\hline
$1$\\\hline
\end{tabular}
\\
\frac{1}{a_{14}} & \frac{1}{a_{14}} & 1 & \smallskip%
\begin{tabular}
[c]{|c|}\hline
$1$\\\hline
\end{tabular}
\\
\frac{1}{a_{14}} & \frac{1}{a_{14}} & \frac{1}{a_{14}} & 1
\end{array}
\right]  ,\quad S_{2}^{(\alpha)}:\left[
\begin{array}
[c]{cccc}%
\smallskip1 & 1 & 1 & 1\\
\smallskip%
\begin{tabular}
[c]{|c|}\hline
$\frac{1}{a_{14}}$\\\hline
\end{tabular}
& 1 &
\begin{tabular}
[c]{|c|}\hline
$1$\\\hline
\end{tabular}
&
\begin{tabular}
[c]{|c|}\hline
$1$\\\hline
\end{tabular}
\\
\smallskip%
\begin{tabular}
[c]{|c|}\hline
$\frac{1}{a_{14}}$\\\hline
\end{tabular}
& \frac{1}{a_{14}} & 1 &
\begin{tabular}
[c]{|c|}\hline
$1$\\\hline
\end{tabular}
\\
\smallskip%
\begin{tabular}
[c]{|c|}\hline
$\frac{1}{a_{14}}$\\\hline
\end{tabular}
& \frac{1}{a_{14}} & \frac{1}{a_{14}} & 1
\end{array}
\right]  \bigskip\\
S_{3}^{(\alpha)}  &  :\left[
\begin{array}
[c]{cccc}%
1 & \smallskip%
\begin{tabular}
[c]{|c|}\hline
$1$\\\hline
\end{tabular}
& 1 & 1\\
\frac{1}{a_{14}} & \smallskip1 & 1 & 1\\%
\begin{tabular}
[c]{|c|}\hline
$\frac{1}{a_{14}}$\\\hline
\end{tabular}
& \smallskip%
\begin{tabular}
[c]{|c|}\hline
$\frac{1}{a_{14}}$\\\hline
\end{tabular}
& 1 &
\begin{tabular}
[c]{|c|}\hline
$1$\\\hline
\end{tabular}
\\%
\begin{tabular}
[c]{|c|}\hline
$\frac{1}{a_{14}}$\\\hline
\end{tabular}
& \smallskip%
\begin{tabular}
[c]{|c|}\hline
$\frac{1}{a_{14}}$\\\hline
\end{tabular}
& \frac{1}{a_{14}} & 1
\end{array}
\right]  ,\quad S_{4}^{(\alpha)}:\left[
\begin{array}
[c]{cccc}%
1 &
\begin{tabular}
[c]{|c|}\hline
$1$\\\hline
\end{tabular}
& \smallskip%
\begin{tabular}
[c]{|c|}\hline
$1$\\\hline
\end{tabular}
& 1\\
\frac{1}{a_{14}} & 1 & \smallskip%
\begin{tabular}
[c]{|c|}\hline
$1$\\\hline
\end{tabular}
& 1\\
\frac{1}{a_{14}} & \frac{1}{a_{14}} & \smallskip1 & 1\\%
\begin{tabular}
[c]{|c|}\hline
$\frac{1}{a_{14}}$\\\hline
\end{tabular}
&
\begin{tabular}
[c]{|c|}\hline
$\frac{1}{a_{14}}$\\\hline
\end{tabular}
& \smallskip%
\begin{tabular}
[c]{|c|}\hline
$\frac{1}{a_{14}}$\\\hline
\end{tabular}
& 1
\end{array}
\right]  .
\end{align*}
The case of $\alpha^{\prime}$ will not be used. The entries in $P_{A,\beta}$
indexed by the sets $S_{k}^{(\beta)}$ are highlighted next:
\begin{align*}
S_{1}^{(\beta)}  &  :\left[
\begin{array}
[c]{cccc}%
1 &
\begin{tabular}
[c]{|c|}\hline
$1$\\\hline
\end{tabular}
& \smallskip%
\begin{tabular}
[c]{|c|}\hline
$a_{24}$\\\hline
\end{tabular}
&
\begin{tabular}
[c]{|c|}\hline
$a_{24}$\\\hline
\end{tabular}
\\
\beta(A) & 1 & \smallskip%
\begin{tabular}
[c]{|c|}\hline
$a_{24}$\\\hline
\end{tabular}
&
\begin{tabular}
[c]{|c|}\hline
$a_{24}$\\\hline
\end{tabular}
\\
\frac{1}{a_{13}} & \frac{1}{a_{13}} & \smallskip1 & \beta(A)\\
\frac{1}{a_{13}} & \frac{1}{a_{13}} & \smallskip%
\begin{tabular}
[c]{|c|}\hline
$1$\\\hline
\end{tabular}
& 1
\end{array}
\right]  ,\quad S_{2}^{(\beta)}:\left[
\begin{array}
[c]{cccc}%
\smallskip1 & 1 & a_{24} & a_{24}\\
\smallskip%
\begin{tabular}
[c]{|c|}\hline
$\beta(A)$\\\hline
\end{tabular}
& 1 &
\begin{tabular}
[c]{|c|}\hline
$a_{24}$\\\hline
\end{tabular}
&
\begin{tabular}
[c]{|c|}\hline
$a_{24}$\\\hline
\end{tabular}
\\
\smallskip%
\begin{tabular}
[c]{|c|}\hline
$\frac{1}{a_{13}}$\\\hline
\end{tabular}
& \frac{1}{a_{13}} & 1 & \beta(A)\\
\smallskip%
\begin{tabular}
[c]{|c|}\hline
$\frac{1}{a_{13}}$\\\hline
\end{tabular}
& \frac{1}{a_{13}} &
\begin{tabular}
[c]{|c|}\hline
$1$\\\hline
\end{tabular}
& 1
\end{array}
\right]  \bigskip\\
S_{3}^{(\beta)}  &  :\left[
\begin{array}
[c]{cccc}%
1 & \smallskip%
\begin{tabular}
[c]{|c|}\hline
$1$\\\hline
\end{tabular}
& a_{24} &
\begin{tabular}
[c]{|c|}\hline
$a_{24}$\\\hline
\end{tabular}
\\
\beta(A) & \smallskip1 & a_{24} &
\begin{tabular}
[c]{|c|}\hline
$a_{24}$\\\hline
\end{tabular}
\\%
\begin{tabular}
[c]{|c|}\hline
$\frac{1}{a_{13}}$\\\hline
\end{tabular}
& \smallskip%
\begin{tabular}
[c]{|c|}\hline
$\frac{1}{a_{13}}$\\\hline
\end{tabular}
& 1 &
\begin{tabular}
[c]{|c|}\hline
$\beta(A)$\\\hline
\end{tabular}
\\
\frac{1}{a_{13}} & \smallskip\frac{1}{a_{13}} & 1 & 1
\end{array}
\right]  ,\quad S_{4}^{(\beta)}:\left[
\begin{array}
[c]{cccc}%
1 & \smallskip%
\begin{tabular}
[c]{|c|}\hline
$1$\\\hline
\end{tabular}
& a_{24} & a_{24}\\
\beta(A) & \smallskip1 & a_{24} & a_{24}\\%
\begin{tabular}
[c]{|c|}\hline
$\frac{1}{a_{13}}$\\\hline
\end{tabular}
& \smallskip%
\begin{tabular}
[c]{|c|}\hline
$\frac{1}{a_{13}}$\\\hline
\end{tabular}
& 1 & \beta(A)\\%
\begin{tabular}
[c]{|c|}\hline
$\frac{1}{a_{13}}$\\\hline
\end{tabular}
& \smallskip%
\begin{tabular}
[c]{|c|}\hline
$\frac{1}{a_{13}}$\\\hline
\end{tabular}
&
\begin{tabular}
[c]{|c|}\hline
$1$\\\hline
\end{tabular}
& 1
\end{array}
\right]  .
\end{align*}
The entries in $P_{A,\beta^{\prime}}$ indexed by the sets $S_{k}%
^{(\beta^{\prime})}$ are highlighted next:
\begin{align*}
S_{1}^{(\beta^{\prime})}  &  :\left[
\begin{array}
[c]{cccc}%
1 & \smallskip%
\begin{tabular}
[c]{|c|}\hline
$\beta^{\prime}(A)$\\\hline
\end{tabular}
&
\begin{tabular}
[c]{|c|}\hline
$a_{13}$\\\hline
\end{tabular}
&
\begin{tabular}
[c]{|c|}\hline
$a_{13}$\\\hline
\end{tabular}
\\
1 & \smallskip1 & a_{13} & a_{13}\\
\frac{1}{a_{24}} & \smallskip%
\begin{tabular}
[c]{|c|}\hline
$\frac{1}{a_{24}}$\\\hline
\end{tabular}
& 1 &
\begin{tabular}
[c]{|c|}\hline
$1$\\\hline
\end{tabular}
\\
\frac{1}{a_{24}} & \smallskip%
\begin{tabular}
[c]{|c|}\hline
$\frac{1}{a_{24}}$\\\hline
\end{tabular}
& \beta^{\prime}(A) & 1
\end{array}
\right]  ,\quad S_{2}^{(\beta^{\prime})}:\left[
\begin{array}
[c]{cccc}%
1 & \beta^{\prime}(A) &
\begin{tabular}
[c]{|c|}\hline
$a_{13}$\\\hline
\end{tabular}
& \smallskip%
\begin{tabular}
[c]{|c|}\hline
$a_{13}$\\\hline
\end{tabular}
\\%
\begin{tabular}
[c]{|c|}\hline
$1$\\\hline
\end{tabular}
& 1 &
\begin{tabular}
[c]{|c|}\hline
$a_{13}$\\\hline
\end{tabular}
& \smallskip%
\begin{tabular}
[c]{|c|}\hline
$a_{13}$\\\hline
\end{tabular}
\\
\frac{1}{a_{24}} & \frac{1}{a_{24}} & 1 & \smallskip%
\begin{tabular}
[c]{|c|}\hline
$1$\\\hline
\end{tabular}
\\
\frac{1}{a_{24}} & \frac{1}{a_{24}} & \beta^{\prime}(A) & \smallskip1
\end{array}
\right]  \bigskip\\
S_{3}^{(\beta^{\prime})}  &  :\left[
\begin{array}
[c]{cccc}%
\smallskip1 & \beta^{\prime}(A) & a_{13} & a_{13}\\
\smallskip%
\begin{tabular}
[c]{|c|}\hline
$1$\\\hline
\end{tabular}
& 1 & a_{13} & a_{13}\\
\smallskip%
\begin{tabular}
[c]{|c|}\hline
$\frac{1}{a_{24}}$\\\hline
\end{tabular}
&
\begin{tabular}
[c]{|c|}\hline
$\frac{1}{a_{24}}$\\\hline
\end{tabular}
& 1 &
\begin{tabular}
[c]{|c|}\hline
$1$\\\hline
\end{tabular}
\\
\smallskip%
\begin{tabular}
[c]{|c|}\hline
$\frac{1}{a_{24}}$\\\hline
\end{tabular}
&
\begin{tabular}
[c]{|c|}\hline
$\frac{1}{a_{24}}$\\\hline
\end{tabular}
& \beta^{\prime}(A) & 1
\end{array}
\right]  ,\quad S_{4}^{(\beta^{\prime})}:\left[
\begin{array}
[c]{cccc}%
1 & \beta^{\prime}(A) & \smallskip%
\begin{tabular}
[c]{|c|}\hline
$a_{13}$\\\hline
\end{tabular}
& a_{13}\\%
\begin{tabular}
[c]{|c|}\hline
$1$\\\hline
\end{tabular}
& 1 & \smallskip%
\begin{tabular}
[c]{|c|}\hline
$a_{13}$\\\hline
\end{tabular}
& a_{13}\\
\frac{1}{a_{24}} & \frac{1}{a_{24}} & \smallskip1 & 1\\%
\begin{tabular}
[c]{|c|}\hline
$\frac{1}{a_{24}}$\\\hline
\end{tabular}
&
\begin{tabular}
[c]{|c|}\hline
$\frac{1}{a_{24}}$\\\hline
\end{tabular}
& \smallskip%
\begin{tabular}
[c]{|c|}\hline
$\beta^{\prime}(A)$\\\hline
\end{tabular}
& 1
\end{array}
\right]  .
\end{align*}
The entries in $P_{A,\gamma}$ indexed by the sets $S_{k}^{(\gamma)}$ are
highlighted next:
\begin{align*}
S_{1}^{(\gamma)}  &  :\left[
\begin{array}
[c]{cccc}%
1 &
\begin{tabular}
[c]{|c|}\hline
$a_{13}$\\\hline
\end{tabular}
&
\begin{tabular}
[c]{|c|}\hline
$a_{13}$\\\hline
\end{tabular}
& \smallskip%
\begin{tabular}
[c]{|c|}\hline
$a_{13}a_{24}$\\\hline
\end{tabular}
\\
\frac{a_{24}}{a_{14}} & 1 & \gamma(A) & \smallskip%
\begin{tabular}
[c]{|c|}\hline
$a_{24}$\\\hline
\end{tabular}
\\
\frac{a_{24}}{a_{14}} &
\begin{tabular}
[c]{|c|}\hline
$1$\\\hline
\end{tabular}
& 1 & \smallskip%
\begin{tabular}
[c]{|c|}\hline
$a_{24}$\\\hline
\end{tabular}
\\
\frac{1}{a_{14}} & \frac{a_{13}}{a_{14}} & \frac{a_{13}}{a_{14}} & \smallskip1
\end{array}
\right]  ,\quad S_{2}^{(\gamma)}:\left[
\begin{array}
[c]{cccc}%
1 & a_{13} & \smallskip%
\begin{tabular}
[c]{|c|}\hline
$a_{13}$\\\hline
\end{tabular}
& a_{13}a_{24}\\%
\begin{tabular}
[c]{|c|}\hline
$\frac{a_{24}}{a_{14}}$\\\hline
\end{tabular}
& 1 & \smallskip%
\begin{tabular}
[c]{|c|}\hline
$\gamma(A)$\\\hline
\end{tabular}
&
\begin{tabular}
[c]{|c|}\hline
$a_{24}$\\\hline
\end{tabular}
\\
\frac{a_{24}}{a_{14}} & 1 & \smallskip1 & a_{24}\\%
\begin{tabular}
[c]{|c|}\hline
$\frac{1}{a_{14}}$\\\hline
\end{tabular}
& \frac{a_{13}}{a_{14}} & \smallskip%
\begin{tabular}
[c]{|c|}\hline
$\frac{a_{13}}{a_{14}}$\\\hline
\end{tabular}
& 1
\end{array}
\right]  \bigskip\\
S_{3}^{(\gamma)}  &  :\left[
\begin{array}
[c]{cccc}%
\smallskip1 & a_{13} & a_{13} & a_{13}a_{24}\\
\smallskip%
\begin{tabular}
[c]{|c|}\hline
$\frac{a_{24}}{a_{14}}$\\\hline
\end{tabular}
& 1 & \gamma(A) &
\begin{tabular}
[c]{|c|}\hline
$a_{24}$\\\hline
\end{tabular}
\\
\smallskip%
\begin{tabular}
[c]{|c|}\hline
$\frac{a_{24}}{a_{14}}$\\\hline
\end{tabular}
&
\begin{tabular}
[c]{|c|}\hline
$1$\\\hline
\end{tabular}
& 1 &
\begin{tabular}
[c]{|c|}\hline
$a_{24}$\\\hline
\end{tabular}
\\
\smallskip%
\begin{tabular}
[c]{|c|}\hline
$\frac{1}{a_{14}}$\\\hline
\end{tabular}
& \frac{a_{13}}{a_{14}} & \frac{a_{13}}{a_{14}} & 1
\end{array}
\right]  ,\quad S_{4}^{(\gamma)}:\left[
\begin{array}
[c]{cccc}%
1 & \smallskip%
\begin{tabular}
[c]{|c|}\hline
$a_{13}$\\\hline
\end{tabular}
&
\begin{tabular}
[c]{|c|}\hline
$a_{13}$\\\hline
\end{tabular}
& a_{13}a_{24}\\
\frac{a_{24}}{a_{14}} & \smallskip1 & \gamma(A) & a_{24}\\
\frac{a_{24}}{a_{14}} & \smallskip%
\begin{tabular}
[c]{|c|}\hline
$1$\\\hline
\end{tabular}
& 1 & a_{24}\\%
\begin{tabular}
[c]{|c|}\hline
$\frac{1}{a_{14}}$\\\hline
\end{tabular}
& \smallskip%
\begin{tabular}
[c]{|c|}\hline
$\frac{a_{13}}{a_{14}}$\\\hline
\end{tabular}
&
\begin{tabular}
[c]{|c|}\hline
$\frac{a_{13}}{a_{14}}$\\\hline
\end{tabular}
& 1
\end{array}
\right]  .
\end{align*}
The entries in $P_{A,\gamma^{\prime}}$ indexed by the sets $S_{k}%
^{(\gamma^{\prime})}$ are highlighted next:
\begin{align*}
S_{1}^{(\gamma^{\prime})}  &  :\left[
\begin{array}
[c]{cccc}%
1 &
\begin{tabular}
[c]{|c|}\hline
$\frac{a_{14}}{a_{24}}$\\\hline
\end{tabular}
& \smallskip%
\begin{tabular}
[c]{|c|}\hline
$\frac{a_{14}}{a_{24}}$\\\hline
\end{tabular}
&
\begin{tabular}
[c]{|c|}\hline
$a_{14}$\\\hline
\end{tabular}
\\
\frac{1}{a_{13}} & 1 & \smallskip%
\begin{tabular}
[c]{|c|}\hline
$1$\\\hline
\end{tabular}
& \frac{a_{14}}{a_{13}}\\
\frac{1}{a_{13}} & \gamma^{\prime}(A) & \smallskip1 & \frac{a_{14}}{a_{13}}\\
\frac{1}{a_{13}a_{24}} &
\begin{tabular}
[c]{|c|}\hline
$\frac{1}{a_{24}}$\\\hline
\end{tabular}
& \smallskip%
\begin{tabular}
[c]{|c|}\hline
$\frac{1}{a_{24}}$\\\hline
\end{tabular}
& 1
\end{array}
\right]  ,\quad S_{2}^{(\gamma^{\prime})}:\left[
\begin{array}
[c]{cccc}%
1 & \frac{a_{14}}{a_{24}} & \frac{a_{14}}{a_{24}} & \smallskip%
\begin{tabular}
[c]{|c|}\hline
$a_{14}$\\\hline
\end{tabular}
\\%
\begin{tabular}
[c]{|c|}\hline
$\frac{1}{a_{13}}$\\\hline
\end{tabular}
& 1 &
\begin{tabular}
[c]{|c|}\hline
$1$\\\hline
\end{tabular}
& \smallskip%
\begin{tabular}
[c]{|c|}\hline
$\frac{a_{14}}{a_{13}}$\\\hline
\end{tabular}
\\%
\begin{tabular}
[c]{|c|}\hline
$\frac{1}{a_{13}}$\\\hline
\end{tabular}
& \gamma^{\prime}(A) & 1 & \smallskip%
\begin{tabular}
[c]{|c|}\hline
$\frac{a_{14}}{a_{13}}$\\\hline
\end{tabular}
\\
\frac{1}{a_{13}a_{24}} & \frac{1}{a_{24}} & \frac{1}{a_{24}} & \smallskip1
\end{array}
\right]  \bigskip\\
S_{3}^{(\gamma^{\prime})}  &  :\left[
\begin{array}
[c]{cccc}%
1 & \smallskip%
\begin{tabular}
[c]{|c|}\hline
$\frac{a_{14}}{a_{24}}$\\\hline
\end{tabular}
& \frac{a_{14}}{a_{24}} &
\begin{tabular}
[c]{|c|}\hline
$a_{14}$\\\hline
\end{tabular}
\\
\frac{1}{a_{13}} & \smallskip1 & 1 & \frac{a_{14}}{a_{13}}\\%
\begin{tabular}
[c]{|c|}\hline
$\frac{1}{a_{13}}$\\\hline
\end{tabular}
& \smallskip%
\begin{tabular}
[c]{|c|}\hline
$\gamma^{\prime}(A)$\\\hline
\end{tabular}
& 1 &
\begin{tabular}
[c]{|c|}\hline
$\frac{a_{14}}{a_{13}}$\\\hline
\end{tabular}
\\
\frac{1}{a_{13}a_{24}} & \smallskip%
\begin{tabular}
[c]{|c|}\hline
$\frac{1}{a_{24}}$\\\hline
\end{tabular}
& \frac{1}{a_{24}} & 1
\end{array}
\right]  ,\quad S_{4}^{(\gamma^{\prime})}:\left[
\begin{array}
[c]{cccc}%
\smallskip1 & \frac{a_{14}}{a_{24}} & \frac{a_{14}}{a_{24}} & a_{14}\\
\smallskip%
\begin{tabular}
[c]{|c|}\hline
$\frac{1}{a_{13}}$\\\hline
\end{tabular}
& 1 &
\begin{tabular}
[c]{|c|}\hline
$1$\\\hline
\end{tabular}
& \frac{a_{14}}{a_{13}}\\
\smallskip%
\begin{tabular}
[c]{|c|}\hline
$\frac{1}{a_{13}}$\\\hline
\end{tabular}
& \gamma^{\prime}(A) & 1 & \frac{a_{14}}{a_{13}}\\
\smallskip%
\begin{tabular}
[c]{|c|}\hline
$\frac{1}{a_{13}a_{24}}$\\\hline
\end{tabular}
&
\begin{tabular}
[c]{|c|}\hline
$\frac{1}{a_{24}}$\\\hline
\end{tabular}
&
\begin{tabular}
[c]{|c|}\hline
$\frac{1}{a_{24}}$\\\hline
\end{tabular}
& 1
\end{array}
\right]  .
\end{align*}

\end{example}

\begin{remark}
\label{remcompl}Let $A,B\in\mathcal{PC}_{n}^{0}$ and $\tau\in\Gamma(A)$. Let
$k_{1},k_{2}\in N,$ $k_{1}\neq k_{2}.$ Because of (\ref{proij}) and the fact
that, for $i\neq j,$ either $(i,j)$ or $(j,i)$ is in $S_{k_{1}}^{(\tau)},$ the
matrix $P_{A,\tau}$ is uniquely determined given $\tau(A)$ and $P_{A,\tau
}[S_{k_{1}}^{(\tau)}].$ Also, $P_{A,\tau}$ is uniquely determined by
$P_{A,\tau}[S_{k_{1}}^{(\tau)}\cup S_{k_{2}}^{(\tau)}],$ as this set of
entries of $P_{A,\tau}$ contains a pair of symmetrically located entries and,
thus, determines $\tau(A).$
\end{remark}

\bigskip

We now give the main result of this section. It shows that the sets
$S_{k}^{(\tau)},$ $k=1,\ldots,n,$ of $\frac{n^{2}-n}{2}$ positions are the
maximal sets in which $P_{A,\tau}$ and $W=\left[  \frac{w_{i}}{w_{j}}\right]
$ can coincide, for an efficient vector $w\in\mathcal{E}_{\tau}(A).$

\bigskip

\begin{theorem}
\label{le2}If $A\in\mathcal{PC}_{n}^{0}$ and $\tau\in\Gamma(A)$, then the
following holds:

\begin{enumerate}
\item for each $k\in N$, there is a $w\in\mathcal{E}_{\tau}(A)$ such that
$P_{A,\tau}[S_{k}^{(\tau)}]=W[S_{k}^{(\tau)}].$ The vector $w$ is the solution
of the inequalities in (\ref{ineqe}), with all of them replaced by equalities,
except the one relating $k$ and the index adjacently preceding it in $\tau$
($w$ is an extreme of $\mathcal{E}_{\tau}(A)$);

\item if $w\in\mathcal{E}_{\tau}(A)$ and $P_{A,\tau}[S]=W[S],$ for some
$\emptyset\neq S\subseteq\mathcal{N}$, then there is a $k\in N$ such that
$S\subseteq S_{k}^{(\tau)}.$
\end{enumerate}
\end{theorem}

\begin{proof}
Let $\tau=\tau_{1}\tau_{2}\cdots\tau_{n}\tau_{1},$ with $\tau_{1}=k.$ Let $w$
be such that%
\[
w_{\tau_{1}}=P_{A,\tau}(\tau_{1},\tau_{2})w_{\tau_{2}}=P_{A,\tau}(\tau
_{1},\tau_{3})w_{\tau_{3}}=\cdots=P_{A,\tau}(\tau_{1},\tau_{n})w_{\tau_{n}}.
\]
Clearly, $w\in\mathcal{E}_{\tau}(A)$.\ Let $(p,q)\in S_{k}^{(\tau)}$. Then
$q\neq k$. If $p=k$, then $\frac{w_{p}}{w_{q}}=\frac{w_{k}}{w_{q}}=P_{A,\tau
}(k,q).$ If $p\neq k,$ then $P_{A,\tau}(k,q)=P_{A,\tau}(k,p)P_{A,\tau}(p,q)$
and we have
\[
\frac{w_{p}}{w_{q}}=\frac{\frac{w_{k}}{P_{A,\tau}(k,p)}}{\frac{w_{k}%
}{P_{A,\tau}(k,q)}}=\frac{P_{A,\tau}(k,q)}{P_{A,\tau}(k,p)}=P_{A,\tau}(p,q).
\]
This proves the first claim.

Suppose that $w\in\mathcal{E}_{\tau}(A)$ and $P_{A,\tau}[S]=W[S]$. Vector $w$
is a solution of the inequalities in (\ref{ineqe}), with at least one, and at
most $n-1,$ of them replaced by equalities. Any set $S_{k}^{(\tau)}$
corresponding to $n-1$ equalities in (\ref{ineqe}) that includes the
aforementioned equalities verifies the second claim.
\end{proof}

\bigskip

For $S\subseteq\mathcal{N}$, we have $P_{A,\tau}[S]=W[S]$ if and only if
$P_{A,\tau}^{(-T)}[S^{\prime}]=W[S^{\prime}]$, in which the set $S^{\prime}$
is the transpose of $S$. Thus, the maximal sets of entries in which $W$
coincides with $P_{A,\tau}^{(-T)}$ are the transposes of the sets
$S_{k}^{(\tau)},$ $k\in N$, that is, the sets $S_{k}^{(\operatorname*{rev}%
\tau)},$ $k\in N.$ So, we focus on the sets of entries that can attain the
lower bound in (\ref{intt01}).

\bigskip

We finally show that, if $P_{A,\tau}\ $and $P_{B,\nu}$ coincide in the
positions indexed by $S_{k}^{(\tau)}$, then $\tau=\nu$ (though not necessarily
does $\tau(A)=\nu(B)$).

\begin{remark}
\label{Rperm}Suppose that the matrices $W=ww^{(-T)}$ and $P_{A,\tau}$,
$\tau\in\Gamma(A),$ coincide in the positions indexed by $S_{k}^{(\tau)}$. If
$Q$ is a permutation matrix, then $(Q^{T}w)(Q^{T}w)^{(-T)}=Q^{T}WQ$ and
$Q^{T}P_{A,\tau}Q$ coincide in the corresponding positions. By Theorem
\ref{le2}, this means that this set of positions is an $S_{k^{\prime}}%
^{(\rho)},$ where $\rho$ is the H-cycle given in Lemma \ref{LpermutLA}. If $D$
is a diagonal matrix, then $(D^{-1}w)(D^{-1}w)^{(-T)}=D^{-1}WD$ and
$D^{-1}P_{A,\tau}D$ coincide in the positions indexed by $S_{k}^{(\tau)}$.
\end{remark}

\begin{theorem}
\label{thsamecycle}Let $A,B\in\mathcal{PC}_{n}^{0}$, $\tau\in\Gamma(A)$,
$\nu\in\Gamma(B)\ $and $k\in N.$ If $P_{A,\tau}[S_{k}^{(\tau)}]=P_{B,\nu
}[S_{k}^{(\tau)}]$, then $\tau=\nu.$
\end{theorem}

\begin{proof}
Taking into account Remark \ref{Rperm}, by a simultaneous monomial similarity
on $A$ and $B,$ suppose that $\tau=123\cdots n1$ and $a_{i,i+1}=1$ for
$i=1,\ldots,n-1.$ Then $a_{n1}=\tau(A).$ Let $\nu=\nu_{1}\nu_{2}\cdots\nu
_{n}\nu_{1},$ with $\nu_{1}=1.$ Suppose that $\nu\neq\tau,$ in order to get a
contradiction. In that case, there is an $i\in\{2,\ldots,n-1\}$ such that
$\nu_{i}>\nu_{i+1}.$ We have $P_{A,\tau}(1,\nu_{i})=P_{A,\tau}(1,\nu
_{i+1})=1.$ Also, $P_{A,\tau}(\nu_{i},\nu_{i+1})=\tau(A).$ \ 

Notice that, if $(p,q)\notin S_{k}^{(\tau)},$ $p\neq q,$ then $(q,p)\in
S_{k}^{(\tau)}$ and, by assumption, $P_{A,\tau}(q,p)=P_{B,\nu}(q,p).$ Since
$P_{A,\tau}(p,q)P_{A,\tau}(q,p)=\tau(A)$ and $P_{B,\nu}(p,q)P_{B,\nu}%
(q,p)=\nu(B)$, we get $P_{A,\tau}(p,q)=sP_{B,\nu}(p,q),$ with $s=\frac
{\tau(A)}{\nu(B)}.$

Since $\nu(B)<1$, we will arrive to a contradiction in each of the possible
situations that could occur.

Case 1: Suppose that $(1,\nu_{i})\in S_{k}^{(\tau)}.$ Then, $(1,\nu
_{i+1}),(\nu_{i+1},\nu_{i})\in S_{k}^{(\tau)}$, implying that $(\nu_{i}%
,\nu_{i+1})\notin S_{k}^{(\tau)}$. Thus,%
\begin{align*}
1  &  =P_{A,\tau}(1,\nu_{i+1})=P_{B,\nu}(1,\nu_{i+1})=P_{B,\nu}(1,\nu
_{i})P_{B,\nu}(\nu_{i},\nu_{i+1})\\
&  =\frac{1}{s}P_{A,\tau}(1,\nu_{i})P_{A,\tau}(\nu_{i},\nu_{i+1})=\frac
{\tau(A)}{s}=\nu(B).
\end{align*}

Case 2: Suppose that $(1,\nu_{i})\notin S_{k}^{(\tau)}.$ We consider two subcases.

Case 2.1: Suppose that $(1,\nu_{i+1})\notin S_{k}^{(\tau)}$. As $(\nu
_{i},1),(\nu_{i+1},1)\in S_{k}^{(\tau)}$, it follows that $(\nu_{i+1},\nu
_{i})\in S_{k}^{(\tau)},$ implying that $(\nu_{i},\nu_{i+1})\notin
S_{k}^{(\tau)}.$ Thus,
\begin{align*}
1  &  =P_{A,\tau}(1,\nu_{i+1})=sP_{B,\nu}(1,\nu_{i+1})=sP_{B,\nu}(1,\nu
_{i})P_{B,\nu}(\nu_{i},\nu_{i+1})\\
&  =\frac{1}{s}P_{A,\tau}(1,\nu_{i})P_{A,\tau}(\nu_{i},\nu_{i+1})=\frac
{\tau(A)}{s}=\nu(B).
\end{align*}

Case 2.2: Suppose that $(1,\nu_{i+1})\in S_{k}^{(\tau)}$. Then $(\nu_{i+1}%
,\nu_{i})\notin S_{k}^{(\tau)},$ as $(1,\nu_{i})\notin S_{k}^{(\tau)},$
implying that $(\nu_{i},\nu_{i+1})\in S_{k}^{(\tau)}.$ Thus,
\begin{align*}
1  &  =P_{A,\tau}(1,\nu_{i+1})=P_{B,\nu}(1,\nu_{i+1})=P_{B,\nu}(1,\nu
_{i})P_{B,\nu}(\nu_{i},\nu_{i+1})\\
&  =\frac{1}{s}P_{A,\tau}(1,\nu_{i})P_{A,\tau}(\nu_{i},\nu_{i+1})=\frac
{\tau(A)}{s}=\nu(B).
\end{align*}

\end{proof}

We note that the first claim in Theorem \ref{thequalpath} follows from Theorem
\ref{thsamecycle}.

\section{Orders occurring among the efficient vectors\label{sorder}}

First we ask which orders occur among vectors in $\mathcal{E}_{\tau}(A)$.

\begin{example}
\label{Exdec0}In Example \ref{exbound} we can see that there is a unique order
for the efficient vectors $w$ for which the matrix $W$ satisfies each of the
inequalities (\ref{b1}), (\ref{b2}) and (\ref{b3}). In case (\ref{b1}), all
vectors are decreasing as $\frac{w_{i}}{w_{i+1}}\geq1$ for $i=1,2,3$. In case
(\ref{b2}), all vectors have the order%
\[
\left[
\begin{array}
[c]{c}%
1\\
2\\
4\\
3
\end{array}
\right]  ,
\]
as $\frac{w_{1}}{w_{2}}\geq1$, $\frac{w_{2}}{w_{4}}\geq1$ and $\frac{w_{4}%
}{w_{3}}\geq1$. Similarly, in case (\ref{b3}), all vectors have the order%
\[
\left[
\begin{array}
[c]{c}%
1\\
3\\
2\\
4
\end{array}
\right]  .
\]
Thus, among all the efficient vectors only the $3$ mentioned orders occur. In
particular, we can observe that, in all efficient vectors, alternative $1$ is
above the other alternatives. The latter fact can also be deduced from the
matrix interval for $W,$ with $w\in\mathcal{E}(A)$, given in Theorem
\ref{tminmax}:%
\begin{equation}
\left[
\begin{array}
[c]{cccc}%
1 & 1 & 1 & 1\\
\frac{1}{7} & 1 & \frac{6}{7} & 1\\
\frac{1}{7} & \frac{1}{7} & 1 & \frac{2}{3}\\
\frac{1}{7} & \frac{1}{7} & \frac{1}{7} & 1
\end{array}
\right]  \leq W\leq\left[
\begin{array}
[c]{cccc}%
1 & 7 & 7 & 7\\
1 & 1 & 7 & 7\\
1 & \frac{7}{6} & 1 & 7\\
1 & 1 & \frac{3}{2} & 1
\end{array}
\right]  . \label{b4}%
\end{equation}
However, these bounds give less information than (\ref{b1}), (\ref{b2}) and
(\ref{b3}). For instance, from (\ref{b1}), (\ref{b2}) and (\ref{b3}), we have
that, if $w\in\mathcal{E}(A)$ and $\frac{w_{2}}{w_{3}}=\frac{6}{7}$, then
$\frac{w_{2}}{w_{4}}\geq2$ and $\frac{w_{1}}{w_{4}}\geq6.$ This does not
follow from (\ref{b4}).
\end{example}

It may happen that all efficient vectors have a unique order, even when
$\#\Gamma(A)>1.$

\begin{example}
\label{Exdec}Let
\[
A=\left[
\begin{array}
[c]{cccc}%
1 & 1 & 3 & 4\\
1 & 1 & 2 & 3\\
\frac{1}{3} & \frac{1}{2} & 1 & 1\\
\frac{1}{4} & \frac{1}{3} & 1 & 1
\end{array}
\right]  .
\]
We have $\Gamma(A)=\{\alpha,\gamma^{\prime}\}$, with $\alpha,$ $\gamma
^{\prime}$ as in (\ref{abg}) and (\ref{abg1})$.$ Moreover, $\alpha(A)=\frac
{1}{2}$ and $\gamma^{\prime}(A)=\frac{8}{9}.$ We have%
\[
P_{A,\alpha}=\left[
\begin{array}
[c]{cccc}%
1 & 1 & 2 & 2\\
\frac{1}{2} & 1 & 2 & 2\\
\frac{1}{4} & \frac{1}{4} & 1 & 1\\
\frac{1}{4} & \frac{1}{4} & \frac{1}{2} & 1
\end{array}
\right]  ,\quad P_{A,\gamma^{\prime}}=\left[
\begin{array}
[c]{cccc}%
1 & \frac{4}{3} & \frac{8}{3} & 4\\
\frac{2}{3} & 1 & 2 & \frac{8}{3}\\
\frac{1}{3} & \frac{4}{9} & 1 & \frac{4}{3}\\
\frac{2}{9} & \frac{1}{3} & \frac{2}{3} & 1
\end{array}
\right]  .
\]
All efficient vectors have (weakly) decreasing order since $P_{A,\alpha
}(i,i+1)\geq1$ and $P_{A,\gamma^{\prime}}(i,i+1)\geq1,$ $i=1,2,3.$ This can
also be deduced from $L_{A}$ using Theorem \ref{tminmax}.
\end{example}

It may happen that a set of alternatives always lies above the remaining
alternatives without any particular order within the set.

\begin{example}
Let
\[
A=\left[
\begin{array}
[c]{cccc}%
1 & \frac{1}{2} & \ast & 4\\
2 & 1 & 6 & \ast\\
\ast & \frac{1}{6} & 1 & \frac{1}{2}\\
\frac{1}{4} & \ast & 2 & 1
\end{array}
\right]  .
\]
(A reciprocal pair of $\ast$'s means an arbitrary reciprocal pair.) Then
$\alpha=12341\in\Gamma(A).$ We have%
\[
P_{A,\alpha}=\left[
\begin{array}
[c]{cccc}%
1 & \frac{1}{2} & 3 & \frac{3}{2}\\
\frac{3}{4} & 1 & 6 & 3\\
\frac{1}{8} & \frac{1}{16} & 1 & \frac{1}{2}\\
\frac{1}{4} & \frac{1}{8} & \frac{3}{4} & 1
\end{array}
\right]  .
\]
We can see that alternatives $1$ and $2$ are above alternatives $3$ and $4$ in
all vectors in $\mathcal{E}_{\alpha}(A)$. On the other hand, since any entry
$i,j$ of $P_{A,\alpha}$ is attained by $\frac{w_{i}}{w_{j}}$ for some
$w\in\mathcal{E}_{\alpha}(A)$, there are $u,v\in\mathcal{E}_{\alpha}(A)$ such
that $u_{1}>u_{2}$ and $v_{2}>v_{1}$. Similarly, there are $u,v\in
\mathcal{E}_{\alpha}(A)$ such that $u_{3}>u_{4}$ and $v_{4}>v_{3}$. In fact,
the possible orders of the vectors in $\mathcal{E}_{\alpha}(A)$ are%
\[
\left[
\begin{array}
[c]{c}%
2\\
1\\
4\\
3
\end{array}
\right]  ,\text{ }\left[
\begin{array}
[c]{c}%
2\\
1\\
3\\
4
\end{array}
\right]  \text{, and }\left[
\begin{array}
[c]{c}%
1\\
2\\
4\\
3
\end{array}
\right]  .
\]
Each is attained by an extreme vector of $\mathcal{E}_{\alpha}(A).$ Vector
$w\in\mathcal{E}_{\alpha}(A)$ such that $W[S_{k}^{(\alpha)}]=P_{A,\alpha
}[S_{k}^{(\alpha)}]$ has the first, second and third orders above for $k=1,$
$4$ and $2$, respectively. Note that the existence of such vectors $w$ is
guaranteed by Theorem \ref{le2}. The decreasing order cannot occur since in
that case, for some $w\in\mathcal{E}_{\alpha}(A)$ we would have $\frac{w_{1}%
}{w_{2}}\geq1$ and $\frac{w_{3}}{w_{4}}\geq1.$ On the other hand, by Theorem
\ref{tt01}, $\frac{w_{3}}{w_{2}}\leq\frac{1}{6}$ and $\frac{w_{1}}{w_{4}}%
\leq4$. This gives the contradiction $6\leq4$.
\end{example}

By Theorems \ref{Rlow} and \ref{tt01}, we have the following.

\begin{theorem}
\label{thwiwj}Let $A\in\mathcal{PC}_{n}^{0},$ $\tau\in\Gamma(A),$ and $i,j\in
N$, $i\neq j$. We have $P_{A,\tau}(i,j)\geq1\ $if and only if $w_{i}\geq
w_{j}$ for all $w\in\mathcal{E}_{\tau}(A).$
\end{theorem}

Let $i_{1},\ldots,i_{p},j_{1},\ldots,j_{q}\in N$ be distinct and $p+q=n$,
$p,q\geq1.$ We have $P_{A,\tau}(i,j)\geq1$ for any $i\in\{i_{1},\ldots
,i_{p}\}$ and $j\in\{j_{1},\ldots,j_{q}\}$ if and only if alternatives
$i_{1},\ldots,i_{p}$ are above alternatives $j_{1},\ldots,j_{q}$ in all
$w\in\mathcal{E}_{\tau}(A)$. If the same holds for the matrix obtained from
$P_{A,\tau}$ by deleting rows and columns $i_{1},\ldots,i_{p}$, then we get a
new set of alternatives that is above the remaining ones. In that way, we may
identify a partition $S_{1},\ldots,S_{t}\ $of $N$ such that any alternative in
$S_{i}$ is (weakly) above any alternative in $S_{i+1}$ for all $w\in
\mathcal{E}_{\tau}(A),$ and no such partition is possible within any set
$S_{i}$ (we may have $t=1$). This gives a partial order for the alternatives
in $\mathcal{E}_{\tau}(A)$.

In particular, if all entries of the $k$th row ($k$th column) of $P_{A,\tau}$
are $\geq1$, then alternative $k$ is above (below) all the remaining
alternatives, for all $w\in\mathcal{E}_{\tau}(A)$.

\bigskip

The next result gives necessary and sufficient conditions for all efficient
vectors in a set $\mathcal{E}_{\tau}(A)$ to have a unique order.

\begin{theorem}
\label{torderset}Let $A\in\mathcal{PC}_{n}^{0}\ $and $\tau\in\Gamma(A)$. The
following are equivalent:

\begin{enumerate}
\item all vectors in $\mathcal{E}_{\tau}(A)$ have the same order;

\item there is a permutation $i_{1}i_{2}\cdots i_{n}$ of $N$ such that
$P_{A,\tau}(i_{t},i_{t+1})\geq1$ for $t=1,\ldots,n-1$;

\item $P_{A,\tau}$ has exactly $\frac{n^{2}-n}{2}$ off-diagonal entries
$\geq1$;

\item for $i,j\in N$, $i>j,$ $P_{A,\tau}(i,j)\geq1$ or $P_{A,\tau}(j,i)\geq1$.
\end{enumerate}
\end{theorem}

\begin{proof}
The proof of the equivalence between 1., 2. and 3. is similar to the proof of
Theorem \ref{Tuo}, given in \cite{FJ8}. Since $P_{A,\tau}(i,j)P_{A,\tau
}(j,i)=\tau(A)<1$, conditions 3. and 4. are equivalent.
\end{proof}

\bigskip

If $w_{i}\geq w_{j}$ for every $w$ in every $\mathcal{E}_{\tau}(A)$, $\tau
\in\Gamma(A)$, then, of course, $w_{i}\geq w_{j}$ for every $w\in
\mathcal{E}(A)$. But then, since $l_{ij}=\min_{\tau\in\Gamma(A)}P_{A,\tau
}(i,j)\geq1$ (Theorem \ref{thwiwj}), the same conclusion could have been drawn
using $L_{A}$ and Theorem \ref{tminmax}. Similarly, facts that result from
observations about orders that are common to every $\mathcal{E}_{\tau}(A)$ may
also be derived from $L_{A}$. For example, if the partial order on the
alternatives in each set is the same, it also follows from $L_{A}$ in a
similar way. See Examples \ref{Exdec0} and \ref{Exdec}.

\section{$\mathcal{E}(A)=\mathcal{E}(B)$ implies $A=B$: background\label{s60}}

We conjecture that, for $A,B\in\mathcal{PC}_{n}$, if $\mathcal{E}%
(A)=\mathcal{E}(B)$, then $A=B$. Here we give some general background and
facts relevant to this conjecture, before proving it in some particular
situations in later sections. First we note a simple case in which the
conjecture is true.

If $A,B\in\mathcal{PC}_{n},$ $A$ is consistent and $\mathcal{E}(A)=\mathcal{E}%
(B)$, then $A=B$. This is so since all columns of $A$ are proportional and
$\mathcal{E}(A)$ consists of the positive multiples of the columns of $A$.
Since the columns of $B$ are in $\mathcal{E}(B)$ \cite{FJ1}, $\mathcal{E}%
(A)=\mathcal{E}(B)$ implies that all columns of $A$ and $B$ are proportional.
They are uniquely determined because the diagonal entries are $1$. So, $A=B$.
Thus, we focus on the conjecture when $A,B\in\mathcal{PC}_{n}^{0}$.

In \cite{FJ8} we have noticed the following.

\begin{theorem}
\label{thLAA}Let $A,B\in\mathcal{PC}_{n}^{0}$. If $\mathcal{E}(A)=\mathcal{E}%
(B),$ then $L_{A}=L_{B}.$
\end{theorem}

It may happen that $L_{A}=L_{B}$ does not imply $\mathcal{E}(A)=\mathcal{E}%
(B)$.

\begin{example}
\label{exn}Let
\[
A=\left[
\begin{array}
[c]{ccccc}%
1 & 1 & \frac{1}{10} & 9 & 500\\
1 & 1 & 1 & 5 & 11\\
10 & 1 & 1 & 1 & 30\\
\frac{1}{9} & \frac{1}{5} & 1 & 1 & 1\\
\frac{1}{500} & \frac{1}{11} & \frac{1}{30} & 1 & 1
\end{array}
\right]  \text{ and }B=\left[
\begin{array}
[c]{ccccc}%
1 & 1 & \frac{1}{10} & 9 & 500\\
1 & 1 & 1 & 5 & 11\\
10 & 1 & 1 & 1 & \mathbf{30.1}\\
\frac{1}{9} & \frac{1}{5} & 1 & 1 & 1\\
\frac{1}{500} & \frac{1}{11} & \frac{\mathbf{1}}{\mathbf{30.1}} & 1 & 1
\end{array}
\right]  .
\]
We have
\[
L_{A}=L_{B}=\left[
\begin{array}
[c]{ccccc}%
1 & \frac{1}{110} & \frac{1}{10} & \frac{1}{10} & \frac{1}{10}\\
\frac{1}{500} & 1 & \frac{1}{1000} & \frac{11}{5000} & \frac{1}{10}\\
\frac{1}{500} & \frac{1}{500} & 1 & \frac{99}{500} & 1\\
\frac{1}{500} & \frac{1}{5000} & \frac{1}{5000} & 1 & \frac{11}{90}\\
\frac{1}{500} & \frac{1}{25000} & \frac{1}{5000} & \frac{1}{5000} & 1
\end{array}
\right]  .
\]
However, $\mathcal{E}(A)\neq\mathcal{E}(B)$. For example,
\[
w=\left[
\begin{array}
[c]{ccccc}%
5 & 5 & \frac{3}{10} & 1 & \frac{1}{100}%
\end{array}
\right]  ^{T}\in\mathcal{E}(A),
\]
but $w\notin\mathcal{E}(B)$. See \cite{blanq2006,FJ2} for a test of efficiency.
\end{example}

\bigskip

We have the following consequence of Theorem \ref{thLAA}, noticed in
\cite{FJ8}.

\begin{corollary}
Let $A,B\in\mathcal{PC}_{n}^{0}$. If $A\neq B$ and $A$ and $B$ are positive
diagonally similar then $\mathcal{E}(A)\neq\mathcal{E}(B).$
\end{corollary}

\begin{proof}
If $B=DAD^{-1},$ in which $D=\operatorname*{diag}(d_{1},\ldots,d_{n})$ is a
positive diagonal matrix, then $L_{B}=DL_{A}D^{-1}.$ Since there are $i,j$
such that $\frac{d_{i}}{d_{j}}\neq1,$ it follows that the $i,j$ entries of
$L_{A}$ and $L_{B}$ are distinct. Thus, the result follows from Theorem
\ref{thLAA}.
\end{proof}

\bigskip When $n=3,$ our conjecture is known to be true \cite{FJ8}. In the
next section we consider a class of matrices in $\mathcal{PC}_{n}$ that
coincides with $\mathcal{PC}_{3}$ when $n=3.$

\section{$\mathcal{E}(A)=\mathcal{E}(B)$ implies $A=B$: the simple perturbed
case\label{s6}}

A simple perturbed consistent (SPC) matrix $A\in\mathcal{PC}_{n}$ is a
reciprocal matrix obtained from a consistent one by modifying one
symmetrically placed pair of reciprocal entries \cite{AbeleBozoki2016}. Such a
matrix (inconsistent) is diagonally similar to a matrix $S_{l,k}(x)$ with
entries $x\neq1$ and $\frac{1}{x}$ in positions $l,k$ and $k,l$ ($l\neq k$),
respectively, and all the remaining entries $1.$ With an additional
permutation similarity, $l$ and $k$ can be made $1$ and $n$, respectively, and
$x$ may be assumed $>1$. (If $A$ is consistent it is diagonally similar to
$J_{n}$.) From Theorem \ref{TeEaconvexunion} (see also \cite{CFF}),
$w\in\mathcal{E}(S_{1,n}(x))$, $x>1$, if and only if%
\[
w_{n}\leq w_{i}\leq w_{1}\leq xw_{n}\text{, }i=2,\ldots,n-1.
\]
If $A=QS_{1,n}(x)Q^{-1},$ in which $Q$ is a monomial matrix, by Lemma
\ref{lsim}, $v\in\mathcal{E}(A)$ if and only if $Q^{-1}v\in\mathcal{E}%
(S_{1,n}(x)).$

\begin{theorem}
\label{Thsimplepert}Let $A,B\in\mathcal{PC}_{n}^{0}$ be SPC matrices. Then,
$L_{A}=L_{B}$ implies $A=B.$ Therefore, $\mathcal{E}(A)=\mathcal{E}(B)$
implies $A=B$.
\end{theorem}

\begin{proof}
The second claim follows from the first one by Theorem \ref{thLAA}. To show
the first claim, by a simultaneous monomial similarity, suppose, without loss
of generality, that $A=S_{1,n}(x),$ $x>1,$ and $B=DS_{l,k}(y)D^{-1},$ in which
$D=\operatorname*{diag}(d_{1},\ldots,d_{n})$ is a positive diagonal matrix and
$y>1$. If $n=3$, we may assume $l=1$ and $k=3$, or $l=3$ and $k=1$.

Note that there are $(n-2)!$ H-cycles in $\Gamma(A)$, which are precisely
those for which $1$ adjacently succeeds $n$. All have cycle products equal to
$\frac{1}{x}$. Similarly, the H-cycles in $\Gamma(B)$ are precisely those for
which $l$ adjacently succeeds $k$. The matrix $L_{A}$ has the entries in row
$1$ and the entries in column $n$ equal to $1$, and all the remaining
off-diagonal entries equal to $\frac{1}{x}.$ The matrix $L_{B}$ has the $l,j$
and $i,k$ entries, with $j\neq l$ and $i\neq k$, equal to $\frac{d_{l}}{d_{j}%
}$ and $\frac{d_{i}}{d_{k}},$ and the remaining entries in positions $i,j$,
with $i\neq j$, equal to $\frac{d_{i}}{d_{j}}\frac{1}{y}$. Suppose that
$L_{A}=L_{B}$. We want to see that $D$ is scalar, $l=1,$ $k=n$ and $x=y.$

Suppose that $k,l\notin\{1,n\}.$ We have $1=L_{A}(1,l)=L_{B}(1,l)=\frac{d_{1}%
}{d_{l}}\frac{1}{y}$ and $\frac{1}{x}=L_{A}(l,1)=L_{B}(l,1)=\frac{d_{l}}%
{d_{1}}.$ Thus $x=y.$ Also, $\frac{1}{x}=L_{A}(l,k)=L_{B}(l,k)=\frac{d_{l}%
}{d_{k}}$ and $\frac{1}{x}=L_{A}(k,l)=L_{B}(k,l)=\frac{d_{k}}{d_{l}}\frac
{1}{x}.$ Thus, $\frac{d_{l}}{d_{k}}=\frac{1}{x}$ and $\frac{d_{l}}{d_{k}}=1$,
a contradiction.

Suppose that $l=n$ and $k\neq1$. We have $1=L_{A}(1,k)=L_{B}(1,k)=\frac{d_{1}%
}{d_{k}}$ and $\frac{1}{x}=L_{A}(k,1)=L_{B}(k,1)=\frac{d_{k}}{d_{1}}\frac
{1}{y},$ implying $y=x.$ Also, for $i\neq1,k,n$, $\frac{1}{x}=L_{A}%
(i,k)=L_{B}(i,k)=\frac{d_{i}}{d_{k}}$ and $\frac{1}{x}=L_{A}(k,i)=L_{B}%
(k,i)=\frac{d_{k}}{d_{i}}\frac{1}{y},$ implying $y=x^{2}.$ Thus, $y=x=1$, a contradiction.

Suppose that $k=1$ and $l\neq n$. We have $1=L_{A}(l,n)=L_{B}(l,n)=\frac
{d_{l}}{d_{n}}$ and $\frac{1}{x}=L_{A}(n,l)=L_{B}(n,l)=\frac{d_{n}}{d_{l}%
}\frac{1}{y},$ implying $y=x.$ Also, for $j\neq1,l,n$, $\frac{1}{x}%
=L_{A}(l,j)=L_{B}(l,j)=\frac{d_{l}}{d_{j}}$ and $\frac{1}{x}=L_{A}%
(j,l)=L_{B}(j,l)=\frac{d_{j}}{d_{l}}\frac{1}{y},$ implying $y=x^{2}.$ Thus,
$y=x=1$, a contradiction.

Suppose that $k=1$ and $l=n$. Equating the last rows, and the first columns,
of $L_{A}$ and $L_{B},$ it follows that $d_{1}=\cdots=d_{n}.$ Then,
$1=L_{A}(1,2)=L_{B}(1,2)=\frac{d_{1}}{d_{2}}\frac{1}{y}=\frac{1}{y}$, a contradiction.

Suppose that $l=1$. Equating the first rows of $L_{A}$ and $L_{B},$ it follows
that $d_{1}=\cdots=d_{n}.$ If $k\neq n$, for any $i\neq1,k,$ we have $\frac
{1}{x}=L_{A}(i,k)=L_{B}(i,k)=\frac{d_{i}}{d_{k}}=1$, a contradiction. Thus,
$k=n\ $and $x=y$, as desired.

Suppose that $k=n$. Equating the last columns of $L_{A}$ and $L_{B},$ it
follows that $d_{1}=\cdots=d_{n}.$ If $l\neq1$, for any $j\neq n,l$ we have
$\frac{1}{x}=L_{A}(l,j)=L_{B}(l,j)=\frac{d_{l}}{d_{j}}=1$, a contradiction.
Thus, $l=1\ $and $x=y$, as desired.
\end{proof}

\bigskip

Any matrix in $\mathcal{PC}_{3}$ is an SPC matrix. Thus, Theorem
\ref{Thsimplepert} implies that, when $A,B\in\mathcal{PC}_{3},$ $\mathcal{E}%
(A)=\mathcal{E}(B)$ implies $A=B$, as noticed in \cite{FJ8}. In Section
\ref{sn4}, we show that this result is valid when $A,B\in\mathcal{PC}_{4}.$
Next we give some important tools for this purpose.

\section{Necessary conditions for $\mathcal{E}(A)=\mathcal{E}(B)$ for general
$n$\label{snec}}

Here we give other necessary conditions for $\mathcal{E}(A)=\mathcal{E}(B)$
that will be important in the next section.

Given matrices $A,B\in M_{n}$ and $S\subseteq N\times N,$ we say that $A[S]$
\emph{dominates} $B[S]$ if $A[S]\leq B[S]$ entry-wise$.$ The dominance is
strict if the inequality is strict for at least one entry. Note that, from
Theorem \ref{thsamecycle}, for $A,B\in\mathcal{PC}_{n}^{0},$ $\tau\in
\Gamma(A)$ and $k\in N$, if there is a $\nu\in\Gamma(A),$ $\nu\neq\tau,$ such
that $P_{A,\nu}[S_{k}^{(\tau)}]$ dominates $P_{B,\tau}[S_{k}^{(\tau)}]$, then
the dominance is strict.

If $A\in\mathcal{PC}_{n}^{0}$ and $\tau\in\Gamma(A),$ we say that $P_{A,\tau
}[S]$ is $(A,S)$\emph{-undominated} if there is no $\nu\in\Gamma(A)$, $\nu
\neq\tau$,$\ $such that $P_{A,\nu}[S]$ dominates $P_{A,\tau}[S]$. Note that,
if $P_{A,\tau}[S_{k}^{(\tau)}]$ is $(A,S_{k}^{(\tau)})$-undominated, then
$w\in\mathbb{R}_{+}^{n}$ such that $W[S_{k}^{(\tau)}]=P_{A,\tau}[S_{k}%
^{(\tau)}]$ is an extreme of $\mathcal{E}_{\tau}(A)$ (by Theorem \ref{le2})
and also of $\mathcal{E}(A),$ as there is no $\nu\in\Gamma(A)$, $\nu\neq\tau$,
such that $w\in\mathcal{E}_{\nu}(A)$ (by Theorem \ref{tt012}).

\bigskip

From Theorems \ref{Rlow} and \ref{tminmax} it follows that $\mathcal{E}%
(A)\subseteq\mathcal{E}(B)$ implies $L_{B}\leq L_{A}$ entry-wise. Next we give
an improvement on this observation.

\begin{theorem}
\label{lemsamerow}Let $A,B\in\mathcal{PC}_{n}^{0}$ and suppose that
$\mathcal{E}(A)\subseteq\mathcal{E}(B)$. Then, for each $k\in N$ and each
$\tau\in\Gamma(A),$ there is $\nu\in\Gamma(B)$ such that $P_{B,\nu}%
[S_{k}^{(\tau)}]\leq P_{A,\tau}[S_{k}^{(\tau)}]$.
\end{theorem}

\begin{proof}
Suppose that $\mathcal{E}(A)\subseteq\mathcal{E}(B)$. By Theorem \ref{le2},
there is $w\in\mathcal{E}(A)$ such that $P_{A,\tau}[S_{k}^{(\tau)}%
]=W[S_{k}^{(\tau)}].$ Since $w\in\mathcal{E}(B)$, there is $\nu\in\Gamma(B)$
such that $w\in\mathcal{E}_{\nu}(B)$. Thus, by Theorem \ref{tt01}, $P_{B,\nu
}\leq W$. Therefore, $P_{B,\nu}[S_{k}^{(\tau)}]\leq W[S_{k}^{(\tau
)}]=P_{A,\tau}[S_{k}^{(\tau)}]$.
\end{proof}

\bigskip

Next we state the main result of this section. It gives conditions under
which, when $\mathcal{E}(A)=\mathcal{E}(B)$, certain H-cycles in $\Gamma(A)$
are also in $\Gamma(B).$ We will use the next two auxiliary results.

\begin{lemma}
\label{ldecrea}Let $A\in\mathcal{PC}_{n}^{0}$ and $\tau\in\Gamma(A).$ If
$w\in\mathcal{E}_{\tau}(A)$ is such that $\frac{w_{i}}{w_{j}}>P_{A,\tau
}(i,j),$ for some $i,j\in N,$ $i\neq j,$ then there is an arbitrarily small
perturbation $w^{\prime}$ of $w$ such that $w^{\prime}\in\mathcal{E}_{\tau
}(A)$ and $\frac{w_{i}^{\prime}}{w_{j}^{\prime}}<\frac{w_{i}}{w_{j}}.$
\end{lemma}

\begin{proof}
Let $\tau=\tau_{1}\cdots\tau_{n}\tau_{1}$, with $\tau_{1}=i.$ Let $k$ be such
that $\tau_{k}=j.$ Since $\frac{w_{i}}{w_{j}}>P_{A,\tau}(i,j)$, there is
$s\in\{2,\ldots,k\}$ such that
\begin{align*}
w_{\tau_{1}}  &  \geq P_{A,\tau}(\tau_{1},\tau_{2})w_{\tau_{2}}\geq\cdots\geq
P_{A,\tau}(\tau_{1},\tau_{s-1})w_{\tau_{s-1}}\\
&  >P_{A,\tau}(\tau_{1},\tau_{s})w_{\tau_{s}}\geq\cdots\geq P_{A,\tau}%
(\tau_{1},\tau_{k})w_{\tau_{k}}\geq\cdots\\
&  \geq P_{A,\tau}(\tau_{1},\tau_{n})w_{\tau_{n}}\geq\tau(A)w_{\tau_{1}}.
\end{align*}
The vector $w^{\prime}$ obtained from $w$ by a small decrement of
$w_{\tau_{s-1}}$ so that $w_{\tau_{s-1}}^{\prime}>P_{A,\tau}(\tau_{s-1}%
,\tau_{s})w_{\tau_{s}}$ verifies the claim.
\end{proof}

\begin{lemma}
\label{lnotea}Let $A\in\mathcal{PC}_{n}^{0},$ $\tau\in\Gamma(A)$ and
$w\in\mathcal{E}_{\tau}(A)$. Let $k\in N$ and $(i,j)\in S_{k}^{(\tau)}.$ If
$P_{A,\tau}[S_{k}^{(\tau)}]=W[S_{k}^{(\tau)}]$ and $P_{A,\tau}[S_{k}^{(\tau
)}]$ is $(A,S_{k}^{(\tau)})$-undominated, then there is a neighborhood
$\mathcal{V}$ of $w$ such that, for any $w^{\prime}\in\mathcal{V}$ with
$\frac{w_{i}^{\prime}}{w_{j}^{\prime}}<\frac{w_{i}}{w_{j}}$, we have
$w^{\prime}\notin\mathcal{E}(A).$
\end{lemma}

\begin{proof}
If $\#\Gamma(A)=1,$ the result follows from Theorem \ref{tt012}. Suppose that
$\#\Gamma(A)>1.$ Let%
\[
\delta=\min_{\rho\in\Gamma(A)\backslash\{\tau\}}\max_{(p,q)\in S_{k}^{(\tau)}%
}\left(  P_{A,\rho}(p,q)-P_{A,\tau}(p,q)\right)  .
\]
Note that $\delta>0,$ since $P_{A,\tau}[S_{k}^{(\tau)}]$ is $(A,S_{k}^{(\tau
)})$-undominated. Any $w^{\prime}$ such that $\frac{w_{i}^{\prime}}%
{w_{j}^{\prime}}<P_{A,\tau}(i,j)$ and $\left\vert \frac{w_{p}^{\prime}}%
{w_{q}^{\prime}}-\frac{w_{p}}{w_{q}}\right\vert <\delta$ for all $(p,q)\in
S_{k}^{(\tau)}$ verifies the claim, as there is no $\nu\in\Gamma(A)$ such that
$w^{\prime}\in\mathcal{E}_{\nu}(A)$.
\end{proof}

\begin{theorem}
\label{t4}Let $A,B\in\mathcal{PC}_{n}^{0}$ and suppose that $\mathcal{E}%
(A)=\mathcal{E}(B).$ Let $\tau\in\Gamma(A)$ and $k\in N.$ If $P_{A,\tau}%
[S_{k}^{(\tau)}]$ is $(A,S_{k}^{(\tau)})$-undominated, then $\tau\in\Gamma(B)$
and $P_{B,\tau}[S_{k}^{(\tau)}]=P_{A,\tau}[S_{k}^{(\tau)}]$.
\end{theorem}

\begin{proof}
By Theorem \ref{le2}, there is $w\in\mathcal{E}_{\tau}(A)$ such that
$P_{A,\tau}[S_{k}^{(\tau)}]=W[S_{k}^{(\tau)}].$ Since $\mathcal{E}%
(A)=\mathcal{E}(B),$ we have $w\in\mathcal{E}(B).$ By Theorem
\ref{TeEaconvexunion}, there is $\nu\in\Gamma(B)$ such that $w\in
\mathcal{E}_{\nu}(B)$. By Theorem \ref{tt01}, $P_{B,\nu}[S_{k}^{(\tau)}]\leq
W[S_{k}^{(\tau)}]=P_{A,\tau}[S_{k}^{(\tau)}].$ Suppose that $P_{B,\nu}%
[S_{k}^{(\tau)}]$ strictly dominates $P_{A,\tau}[S_{k}^{(\tau)}].$ Then, there
is $(i,j)\in S_{k}^{(\tau)}$ such that $P_{B,\nu}(i,j)<P_{A,\tau}%
(i,j)=\frac{w_{i}}{w_{j}}.$ By Lemma \ref{ldecrea}, there is a vector
$w^{\prime}\in\mathcal{E}_{\nu}(B),$ arbitrarily close to $w,$ such that
$\frac{w_{i}^{\prime}}{w_{j}^{\prime}}<\frac{w_{i}}{w_{j}}.$ Since $P_{A,\tau
}[S_{k}^{(\tau)}]$ is $(A,S_{k}^{(\tau)})$-undominated, by Lemma \ref{lnotea},
$w^{\prime}\notin\mathcal{E}(A)$, a contradiction since $\mathcal{E}%
(A)=\mathcal{E}(B).$ Thus, $P_{B,\nu}[S_{k}^{(\tau)}]=P_{A,\tau}[S_{k}%
^{(\tau)}].$ By Theorem \ref{thsamecycle}, $\nu=\tau$. Thus the claim follows.
\end{proof}

Recall that $P_{B,\tau}[S_{k}^{(\tau)}]=P_{A,\tau}[S_{k}^{(\tau)}]$ does not
imply $P_{B,\tau}=P_{A,\tau},$ unless $\tau(A)=\tau(B).$

\begin{corollary}
\label{lauxx}Let $A,B\in\mathcal{PC}_{n}^{0}$. Suppose that $\ $there are
$\tau\in\Gamma(A)$ and $k_{1},k_{2}\in N$, $k_{1}\neq k_{2}$, such that
$P_{A,\tau}[S_{k_{1}}^{(\tau)}]$ is $(A,S_{k_{1}}^{(\tau)})$-undominated and
$P_{A,\tau}[S_{k_{2}}^{(\tau)}]$ is $(A,S_{k_{2}}^{(\tau)})$-undominated. If
$\mathcal{E}(A)=\mathcal{E}(B),$ then $\tau\in\Gamma(B)$ and $P_{A,\tau
}=P_{B,\tau}.$ Thus, the sequence of entries in $A$ and in $B$ along $\tau$
coincide, that is, $\mathcal{E}_{\tau}(A)=\mathcal{E}_{\tau}(B).$
\end{corollary}

\begin{proof}
By Theorem \ref{t4}, if $\mathcal{E}(A)=\mathcal{E}(B)$, then $\tau\in
\Gamma(B)$ and $P_{B,\tau}[S_{k_{1}}^{(\tau)}]=P_{A,\tau}[S_{k_{1}}^{(\tau)}]$
and $P_{B,\tau}[S_{k_{2}}^{(\tau)}]=P_{A,\tau}[S_{k_{2}}^{(\tau)}]$. By Remark
\ref{remcompl}, $P_{B,\tau}=P_{A,\tau}.$ By Theorem \ref{thequalpath}, the
entries in $A$ and $B$ along $\tau$ coincide.
\end{proof}

\bigskip

\begin{remark}
Theorem \ref{t4} says that, if $\mathcal{E}(A)=\mathcal{E}(B),$ and $w$ is an
extreme of $\mathcal{E}(A)$ lying in a subset $\mathcal{E}_{\tau}(A)$, with
$\tau\in\Gamma(A),$ and satisfying the undominance condition, then $\tau
\in\Gamma(B)$ and $w$, which is also an extreme of $\mathcal{E}(B)$, lies in
$\mathcal{E}_{\tau}(B)$. By Corollary \ref{lauxx}, if there are two such
extremes of $\mathcal{E}(A)$ lying in a subset $\mathcal{E}_{\tau}(A)$, then
$\tau\in\Gamma(B)$ and $\mathcal{E}_{\tau}(B)=\mathcal{E}_{\tau}(A)$.
\end{remark}

\bigskip

We now can give a condition under which $\mathcal{E}(A)=\mathcal{E}(B)$
implies $A=B$. \bigskip

\begin{corollary}
Let $A,B\in\mathcal{PC}_{n}^{0}$. Suppose that the entries of a matrix in
$\mathcal{PC}_{n}$ along the H-cycles in $\Gamma(A)$ determine the matrix
(this happens if $\#\Gamma(A)=\frac{(n-1)!}{2}$). Suppose that for each
$\tau\in\Gamma(A)$ there are $k_{1},k_{2}\in N$, $k_{1}\neq k_{2}$, such that
$P_{A,\tau}[S_{k_{1}}^{(\tau)}]$ is $(A,S_{k_{1}}^{(\tau)})$-undominated and
$P_{A,\tau}[S_{k_{2}}^{(\tau)}]$ is $(A,S_{k_{2}}^{(\tau)})$-undominated. If
$\mathcal{E}(A)=\mathcal{E}(B),$ then $A=B.$
\end{corollary}

\begin{proof}
If $\mathcal{E}(A)=\mathcal{E}(B),$ by Corollary \ref{lauxx}, for each
$\tau\in\Gamma(A)$, we have that $\tau\in\Gamma(B)$ and the sequence of
entries in $A$ and in $B$ along $\tau$ coincide. Since, by hypothesis, the
entries of $A$ and $B$ along the H-cycles in $\Gamma(A)\subseteq\Gamma(B)$
determine $A$ and $B$, respectively, and those entries coincide in both
matrices, it follows that $A=B$.
\end{proof}

\section{$\mathcal{E}(A)=\mathcal{E}(B)$ implies $A=B$: the case
$n=4$\label{sn4}}

Here, we show that $\mathcal{E}(A)=\mathcal{E}(B)$ implies $A=B$, when
$A,B\in\mathcal{PC}_{4}^{0}.$ In this case, there are at most $3$ H-cycles in
$\Gamma(A)$.

\bigskip

We first show that, if $A,B\in\mathcal{PC}_{4}^{0}$ and $\mathcal{E}%
(A)=\mathcal{E}(B),$ then the minimum cycle products in $A$ and $B$ are the
same. Given $A\in\mathcal{PC}_{4}^{0}$, we denote by $\theta_{A}$ an H-cycle
$\tau$ such that $\tau(A)=\min_{\rho\in\Gamma(A)}\rho(A).$ The H-cycle
$\theta_{A}$ may not be uniquely determined, but the numerical value
$\theta_{A}(A)=\tau(A)$ is uniquely determined (is the minimum cycle product
in $A$).

\begin{theorem}
\label{cequalprod}Let $A,B\in\mathcal{PC}_{4}^{0}.$ If $\mathcal{E}%
(A)=\mathcal{E}(B)$ then $\theta_{A}(A)=\theta_{B}(B).$
\end{theorem}

\begin{proof}
Without loss of generality, suppose that $\theta_{A}(A)\leq\theta_{B}(B).$ By
Theorem \ref{lemsamerow}, and because $\#\Gamma(B)<4$, there are $k_{1}%
,k_{2}\in\{1,2,3,4\},$ $k_{1}\neq k_{2},$ and $\rho\in\Gamma(B)$ such that
$P_{B,\rho}[S_{k_{1}}^{(\theta_{A})}\cup S_{k_{2}}^{(\theta_{A})}]\leq
P_{A,\theta_{A}}[S_{k_{1}}^{(\theta_{A})}\cup S_{k_{2}}^{(\theta_{A})}]$. By
Remark \ref{tequalprod}, $\rho(B)\leq\theta_{A}(A)$. Since $\theta_{B}%
(B)\leq\rho(B)$ and, by assumption, $\theta_{A}(A)\leq\theta_{B}(B)$, we get
$\theta_{B}(B)=\theta_{A}(A).$
\end{proof}

\bigskip

To prove that $\mathcal{E}(A)=\mathcal{E}(B)$ implies $A=B$, first we show
that, if $\mathcal{E}(A)=\mathcal{E}(B),$ then there is a common H-cycle in
$A$ and $B$ with the same entries along it in both $A$ and $B$ and with
minimal cycle product (which is equal in $A$ and $B$ by Theorem
\ref{cequalprod}). Based on this, we prove the result assuming that $A$ and
$B$ are in a special form attained by monomial similarity.

\bigskip

\textbf{Proof of the existence of a common H-cycle with minimum product}

\bigskip

We consider the case in which the minimum cycle product in $A$ occurs just for
one H-cycle in $\Gamma(A)$, and the case in which it occurs for more than one
H-cycle in $\Gamma(A)$. The following lemma states that, in the former case,
there are two extremes of the set $\mathcal{E}_{\tau}(A)$ associated with the
minimum cycle product that are not in the other sets $\mathcal{E}_{\nu}(A)$
(so they are extremes of $\mathcal{E}(A)$).

\begin{lemma}
\label{t3}Let $A\in\mathcal{PC}_{4}^{0}$ with $\#\Gamma(A)>1$. Suppose that
$\theta_{A}(A)<\tau(A)$ for any $\tau\in\Gamma(A),$ $\tau\neq\theta_{A}.$
Then, there are $k_{1},k_{2}\in\{1,2,3,4\}$, $k_{1}\neq k_{2},$ such that
$P_{A,\theta_{A}}[S_{k_{1}}^{(\theta_{A})}]$ is $(A,S_{k_{1}}^{(\theta_{A})}%
)$-undominated and $P_{A,\theta_{A}}[S_{k_{2}}^{(\theta_{A})}]$ is
$(A,S_{k_{2}}^{(\theta_{A})})$-undominated.
\end{lemma}

\begin{proof}
Suppose that there are at least $3$ $k$'s for which there is $\tau\in
\Gamma(A)$, $\tau\neq\theta_{A}$, such that $P_{A,\tau}[S_{k}^{(\theta_{A})}]$
dominates $P_{A,\theta_{A}}[S_{k}^{(\theta_{A})}]$ (the dominance is strict by
Theorem \ref{thsamecycle})$.$ Then, there are $k_{1},k_{2}\in\{1,2,3,4\},$
$k_{1}\neq k_{2},$ and $\nu\in\Gamma(A),$ $\nu\neq\theta_{A},$ such that
$P_{A,\nu}[S_{k_{1}}^{(\theta_{A})}\cup S_{k_{2}}^{(\theta_{A})}]$ dominates
$P_{A,\theta_{A}}[S_{k_{1}}^{(\theta_{A})}\cup S_{k_{2}}^{(\theta_{A})}]$. By
Remark \ref{tequalprod}, $\nu(A)\leq\theta_{A}(A)$. Then, $\nu(A)=\theta
_{A}(A)$ and, because of the hypothesis, $\nu=\theta_{A},$ a contradiction.
\end{proof}

\bigskip

When the minimum cycle product occurs for at least two H-cycles in
$\Gamma(A),$ then the conclusion in Lemma \ref{t3} may not hold.

\begin{example}
Let
\[
A=\left[
\begin{array}
[c]{cccc}%
1 & 1 & 1 & a\\
1 & 1 & 1 & 1\\
1 & 1 & 1 & 1\\
\frac{1}{a} & 1 & 1 & 1
\end{array}
\right]  ,
\]
$a>1.$ We have $\Gamma(A)=\{\alpha,\gamma\},$ with $\alpha=12341$ and
$\gamma=13241.$ Moreover, $\alpha(A)=\gamma(A)=\frac{1}{a}$. We have (see
Example \ref{ex4by4})
\[
P_{A,\alpha}=\left[
\begin{array}
[c]{cccc}%
1 & 1 & 1 & 1\\
\frac{1}{a} & 1 & 1 & 1\\
\frac{1}{a} & \frac{1}{a} & 1 & 1\\
\frac{1}{a} & \frac{1}{a} & \frac{1}{a} & 1
\end{array}
\right]  \text{ and }P_{A,\gamma}=\left[
\begin{array}
[c]{cccc}%
1 & 1 & 1 & 1\\
\frac{1}{a} & 1 & \frac{1}{a} & 1\\
\frac{1}{a} & 1 & 1 & 1\\
\frac{1}{a} & \frac{1}{a} & \frac{1}{a} & 1
\end{array}
\right]  .
\]
The matrix $P_{A,\alpha}[S_{k_{1}}^{(\alpha)}]$ is $(A,S_{k_{1}}^{(\alpha)}%
)$-undominated if and only if $k_{1}=3$; the matrix $P_{A,\gamma}[S_{k_{2}%
}^{(\gamma)}]$ is $(A,S_{k_{2}}^{(\gamma)})$-undominated if and only if
$k_{2}=2$.

The extremes of $\mathcal{E}_{\alpha}(A)$ are (projectively) $(1,1,1,1),$
$(a,a,a,1),$ $(a,1,1,1),$ and $(a,a,1,1)$, associated with $S_{k}^{(\alpha)}$
for $k=1$, $4$, $2$ and $3$, respectively$.$ The extremes of $\mathcal{E}%
_{\gamma}(A)$ are $(1,1,1,1),$ $(a,a,a,1),$ $(a,1,1,1)$ and $(a,1,a,1)$,
associated with $S_{k}^{(\gamma)}$ for $k=1$, $4$, $3$ and $2$, respectively.
\end{example}

\begin{lemma}
\label{ldom3}Let $A\in\mathcal{PC}_{4}^{0}$. Suppose that there are $\tau
,\nu\in\Gamma(A),$ $\tau\neq\nu$, such that $\tau(A)=\nu(A)=\theta_{A}(A).$ If
there are $3$ distinct $k$'s in $\{1,2,3,4\}$ such that $P_{A,\nu}%
[S_{k}^{(\tau)}]$ dominates $P_{A,\tau}[S_{k}^{(\tau)}]$, then there is an
H-cycle product in $A$ equal to $1.$
\end{lemma}

\begin{proof}
Without loss of generality, we may assume that $\tau=12341$ and $a_{i,i+1}=1,$
$i=1,2,3.$ Suppose that there are $3$ distinct $k$'s such that $P_{A,\nu
}[S_{k}^{(\tau)}]$ dominates $P_{A,\tau}[S_{k}^{(\tau)}]$. Note that the union
of all such sets $S_{k}^{(\tau)}$ contains all pairs of symmetrically placed
positions, except (at most) one (it indexes $3$ rows and $3$ columns). Then,
except for one pair of symmetrically placed entries ($P_{A,\nu}\neq P_{A,\tau
}$, by Theorem \ref{thequalpath}), all entries coincide in $P_{A,\nu}$ and
$P_{A,\tau}$. In fact, by the dominance, each entry in a pair of symmetrically
placed entries in $P_{A,\nu},$ indexed by the union of the $S_{k}^{(\tau)}$'s,
is less than or equal to the corresponding entry in $P_{A,\tau}$, and the
product of the entries in each pair is $\nu(A)=\tau(A)$ in both $P_{A,\nu}$
and $P_{A,\tau}$.

Taking into account Example \ref{ex4by4} (and using the notation there), we
next consider the cases in which $P_{A,\nu}$ and $P_{A,\tau}$ coincide in all
entries except in those placed in one pair of symmetric positions. Recall that
$\tau$ is $\alpha$ as in (\ref{abg}).

\begin{itemize}
\item If $\nu=\beta$, then $P_{A,\nu}$ and $P_{A,\tau}$ differ in the entries
in positions $(3,4)$ and $(4,3).$ Moreover, $a_{24}=1$ and $a_{13}=a_{14},$
implying $\gamma(A)=\gamma^{\prime}(A)=1.$

\item If $\nu=\beta^{\prime}$, then $P_{A,\nu}$ and $P_{A,\tau}$ differ in the
entries in positions $(1,2)$ and $(2,1).$ Moreover, $a_{24}=a_{14}$ and
$a_{13}=1,$ implying $\gamma(A)=\gamma^{\prime}(A)=1.$

\item If $\nu=\gamma$, then $P_{A,\nu}$ and $P_{A,\tau}$ differ in the entries
in positions $(2,3)$ and $(3,2).$ Moreover, $a_{13}=a_{24}=1,$ implying
$\beta(A)=\beta^{\prime}(A)=1.$

\item If $\nu=\gamma^{\prime}$, then $P_{A,\nu}$ and $P_{A,\tau}$ differ in
the entries in positions $(1,4)$ and $(4,1).$ Moreover, $a_{13}=a_{14}%
=a_{24},$ implying $\beta(A)=\beta^{\prime}(A)=1.$
\end{itemize}
\end{proof}

\begin{lemma}
\label{laux4}Let $A\in\mathcal{PC}_{4}^{0}$ and suppose that $\#\Gamma(A)=3.$
Moreover, suppose that there are at least two H-cycles in $\Gamma(A)$ whose
cycle products in $A$ are $\theta_{A}(A)$. If $\nu\in\Gamma(A)$ is such that
$\nu(A)=\theta_{A}(A),$ then there is a $k\in\{1,2,3,4\}$ such that $P_{A,\nu
}[S_{k}^{(\nu)}]$ is $(A,S_{k}^{(\nu)})$-undominated.
\end{lemma}

\begin{proof}
Let $\Gamma(A)=\{\nu,\tau,\rho\}$ with $\nu(A)=\tau(A)=\theta_{A}(A).$ Without
loss of generality, we may assume that $\nu=12341$ and $a_{i,i+1}=1,$
$i=1,2,3.$ Then, $a_{41}=\frac{1}{a_{14}}=\theta_{A}(A)<1.$

Suppose that $\rho(A)>\theta_{A}(A)$ and that there is no $k$ such that
$P_{A,\nu}[S_{k}^{(\nu)}]$ is $(A,S_{k}^{(\nu)})$-undominated. Note that there
is no $\pi\in\Gamma(A)\backslash\{\nu\}$ such that, for all $k\in\{1,2,3,4\},$
$P_{A,\pi}[S_{k}^{(\nu)}]$ dominates $P_{A,\nu}[S_{k}^{(\nu)}]$, as otherwise
$P_{A,\pi}\leq P_{A,\nu},$ implying $P_{A,\pi}=P_{A,\nu},$ as $\nu
(A)=\theta_{A}(A).$ This is impossible by Theorem \ref{thequalpath}, since
$\pi\neq\nu.$ Thus, there is exactly one $k\in\{1,2,3,4\}$ such that
$P_{A,\rho}[S_{k}^{(\nu)}]$ strictly dominates $P_{A,\nu}[S_{k}^{(\nu)}]$, as,
if there were at least two, we would have $\rho(A)\leq\theta_{A}(A),$ a
contradiction. Therefore, there are $3$ $k$'s such that $P_{A,\tau}%
[S_{k}^{(\nu)}]$ strictly dominates $P_{A,\nu}[S_{k}^{(\nu)}].$ By Lemma
\ref{ldom3}, this contradicts the fact that $\#\Gamma(A)=3.$ Thus, there is
$k$ such that $P_{A,\nu}[S_{k}^{(\nu)}]$ is $(A,S_{k}^{(\nu)})$-undominated.

Now suppose that $\rho(A)=\theta_{A}(A)$, that is, all H-cycles in $\Gamma(A)$
have the same product in $A$. We consider all possible sets $\Gamma(A)$. We
use the notation in Example \ref{ex4by4} (see also Example \ref{exSa}). Recall
that $\nu=\alpha.$

\begin{itemize}
\item If $\Gamma(A)=\{\alpha,\beta,\gamma\}$ then $a_{13}a_{24}=1$ and
$\frac{1}{a_{13}}=\frac{1}{\sqrt{a_{14}}}>\frac{1}{a_{14}}.$ In this case,
$P_{A,\nu}[S_{3}^{(\nu)}]$ is $(A,S_{3}^{(\nu)})$-undominated.

\item If $\Gamma(A)=\{\alpha,\beta,\gamma^{\prime}\}$ then $a_{13}%
a_{24}=a_{14}^{2}$ and $\frac{a_{24}}{a_{13}}=\frac{1}{a_{14}}.$ This implies
$a_{24}=\sqrt{a_{14}}>1.$ In this case, $P_{A,\nu}[S_{1}^{(\nu)}]$ is
$(A,S_{1}^{(\nu)})$-undominated.

\item If $\Gamma(A)=\{\alpha,\beta^{\prime},\gamma\}$ then $a_{13}a_{24}=1$
and $\frac{a_{24}}{a_{13}}=a_{14}.$ This implies $\frac{1}{a_{24}}=\frac
{1}{\sqrt{a_{14}}}>\frac{1}{a_{14}}.$ In this case, $P_{A,\nu}[S_{3}^{(\nu)}]$
is $(A,S_{3}^{(\nu)})$-undominated.

\item If $\Gamma(A)=\{\alpha,\beta^{\prime},\gamma^{\prime}\}$ then
$a_{13}a_{24}=a_{14}^{2}$ and $\frac{a_{24}}{a_{13}}=a_{14}.$ This implies
$a_{13}=\sqrt{a_{14}}>1.$ In this case, $P_{A,\nu}[S_{1}^{(\nu)}]$ is
$(A,S_{1}^{(\nu)})$-undominated.
\end{itemize}
\end{proof}

\bigskip

Now we can state the claim that allows us to consider both $A$ and $B$ in a
convenient form attained by monomial similarity.

\begin{theorem}
\label{th4}Let $A,B\in\mathcal{PC}_{4}^{0}$. If $\mathcal{E}(A)=\mathcal{E}%
(B),$ then $\theta_{A}(A)=\theta_{B}(B)$ and there is a $\nu\in\Gamma
(A)\cap\Gamma(B)$ such that $\nu(A)=\theta_{A}(A)=\theta_{B}(B)=\nu(B)$ and
the sequence of entries in $A$ and in $B$ along $\nu\ $ coincide.
\end{theorem}

\begin{proof}
By Theorem \ref{cequalprod}, $\theta_{B}(B)=\theta_{A}(A).$ By interchanging
the roles of $A\ $and $B$, it is enough to consider the following cases.

Case 1: Suppose that $\#\Gamma(A)=1$ or $\theta_{A}(A)<\tau(A)$ for any
$\tau\in\Gamma(A),$ $\tau\neq\theta_{A}$. In the former case, the claim
follows from Corollary \ref{lauxx}, and in the latter case it follows from
Lemma \ref{t3} and Corollary \ref{lauxx}.

Case 2: Suppose that there are two H-cycles in $A$ and in $B$ whose cycle
products are $\theta_{A}(A)=\theta_{B}(B)$.

Case 2.1: Suppose that $\#\Gamma(A)=3.$ Let $\Gamma(A)=\{\pi,\tau,\rho\}$ with
$\pi(A)=\tau(A)=\theta_{A}(A).$ By Lemma \ref{laux4}, there are $k_{1},k_{2}$
such that $P_{A,\pi}[S_{k_{1}}^{(\pi)}]$ is $(A,S_{k_{1}}^{(\pi)}%
)$-undominated and $P_{A,\tau}[S_{k_{2}}^{(\tau)}]$ is $(A,S_{k_{2}}^{(\tau
)})$-undominated$.$ By Theorem \ref{t4}, $\pi,\tau\in\Gamma(B),$ $P_{A,\pi
}[S_{k_{1}}^{(\pi)}]=P_{B,\pi}[S_{k_{1}}^{(\pi)}]$ and $P_{A,\tau}[S_{k_{2}%
}^{(\tau)}]=P_{B,\tau}[S_{k_{2}}^{(\tau)}].$ Taking into account the
hypothesis, either $\pi(B)=\theta_{B}(B)$ or $\tau(B)=\theta_{B}(B).$ In the
first case $P_{A,\pi}=P_{B,\pi}$ and, taking into account Theorem
\ref{thequalpath}, $\nu=\pi$ satisfies the claim; otherwise $P_{A,\tau
}=P_{B,\tau}$ and $\nu=\tau$ satisfies the claim.

Case 2.2: Suppose that $\#\Gamma(A)=\#\Gamma(B)=2.$ Then $\Gamma
(A)=\{\tau,\rho\}$ with $\tau(A)=\rho(A)=\theta_{A}(A).$ Then, for any $\pi
\in\Gamma(A),$ there is $k$ such that $P_{A,\pi}[S_{k}^{(\pi)}]$ is
$(A,S_{k}^{(\pi)})$-undominated, as otherwise $P_{A,\pi^{\prime}}\leq
P_{A,\pi}$ for $\pi^{\prime}\in\Gamma(A),$ $\pi\neq\pi^{\prime}.$ In that
case, because $\pi(A)=\pi^{\prime}(A),$ we would have $P_{A,\pi^{\prime}%
}=P_{A,\pi},$ implying $\pi=\pi^{\prime}$ by Theorem \ref{thequalpath}, a
contradiction. Thus, by Theorem \ref{t4}, if $\mathcal{E}(A)=\mathcal{E}(B)$
then $\Gamma(A)=\Gamma(B).$ Taking into account the hypothesis, for any
$\pi\in\Gamma(A)=\Gamma(B)$ we have $\theta_{A}(A)=\theta_{B}(B)=\pi(B).$ By
Theorem \ref{t4} and Remark \ref{remcompl}, $P_{A,\pi}=P_{B,\pi}$. Thus, the
claim follows from Theorem \ref{thequalpath}. (In fact, the same entries in
$A$ and $B$ along two distinct H-cycles$,$ none of which is the reverse of the
other, implies $A=B$.)
\end{proof}

\bigskip

\textbf{Proof of the main result}

\bigskip

Consider the matrices in $\mathcal{PC}_{4}^{0}$,
\begin{equation}
A=\left[
\begin{array}
[c]{cccc}%
1 & 1 & a_{13} & a_{14}\\
1 & 1 & 1 & a_{24}\\
\frac{1}{a_{13}} & 1 & 1 & 1\\
\frac{1}{a_{14}} & \frac{1}{a_{24}} & 1 & 1
\end{array}
\right]  \text{ and }B=\left[
\begin{array}
[c]{cccc}%
1 & 1 & b_{13} & a_{14}\\
1 & 1 & 1 & b_{24}\\
\frac{1}{b_{13}} & 1 & 1 & 1\\
\frac{1}{a_{14}} & \frac{1}{b_{24}} & 1 & 1
\end{array}
\right]  . \label{AB}%
\end{equation}
Then, for $\alpha=12341$, we have $\alpha(A)=\alpha(B)=\frac{1}{a_{14}}$.
Suppose that $a_{14}>1$ and $\theta_{A}(A)=\theta_{B}(B)=\frac{1}{a_{14}}$.
Each of $\Gamma(A)$ and $\Gamma(B)$ is contained in one of the sets
\[
\{\alpha,\beta,\gamma\},\text{ }\quad\{\alpha,\beta,\gamma^{\prime}\},\text{
}\quad\{\alpha,\beta^{\prime},\gamma\}\text{,}\quad\text{or }\{\alpha
,\beta^{\prime},\gamma^{\prime}\},
\]
with $\beta,\beta^{\prime},\gamma,\gamma^{\prime}$ as in (\ref{abg}) and
(\ref{abg1}). The matrices $P_{B,\rho},$ $\rho\in\Gamma(B)$, can be obtained
from the matrices $P_{A,\rho}$ in Example \ref{ex4by4} by replacing $a_{13}$
and $a_{24}$ by $b_{13}$ and $b_{24}$, respectively.

In the next lemmas, Examples \ref{ex4by4} and \ref{exSa} are helpful.

\begin{lemma}
\label{lB}Let $A,B\in\mathcal{PC}_{4}^{0}$ be as in (\ref{AB}), with
$\theta_{A}(A)=\theta_{B}(B)=\frac{1}{a_{14}}$. If $\mathcal{E}(A)=\mathcal{E}%
(B)$ and $\beta\in\Gamma(A)$, then $A=B.$
\end{lemma}

\begin{proof}
Suppose that $\beta\in\Gamma(A).$ Then, $a_{24}<a_{13}$, as $\beta(A)<1.$
Since $\mathcal{E}(A)\subseteq\mathcal{E}(B),$ by Theorem \ref{lemsamerow},
there is $\rho\in\Gamma(B)$ such that $P_{B,\rho}[S_{3}^{(\beta)}]\leq
P_{A,\beta}[S_{3}^{(\beta)}].$ We will see that this implies $\rho=\beta.$

\begin{itemize}
\item Since $\beta(A)<1,$ it follows that $P_{B,\alpha}[S_{3}^{(\beta)}]$ does
not dominate $P_{A,\beta}[S_{3}^{(\beta)}]$.

\item Suppose that $\gamma\in\Gamma(B)$. Since $\gamma(B)\geq\alpha(B),$ we
have $b_{13}b_{24}\geq1.$ Then, $P_{B,\gamma}[S_{3}^{(\beta)}]$ does not
dominate $P_{A,\beta}[S_{3}^{(\beta)}],$ since $b_{13}\leq1$ and $b_{24}%
\leq\beta(A)<1$ do not occur simultaneously.

\item Suppose that $\beta^{\prime}\in\Gamma(B)$. Since $\beta(A)<1,$ it
follows that $P_{B,\beta^{\prime}}[S_{3}^{(\beta)}]$ does not dominate
$P_{A,\beta}[S_{3}^{(\beta)}]$.

\item Suppose that $\gamma^{\prime}\in\Gamma(B)$. If $\frac{b_{24}}{a_{14}%
}\geq1$, then $\frac{a_{14}}{b_{13}}\geq1$, as, from $\gamma^{\prime}%
(B)\geq\alpha(B)$, we get $\frac{a_{14}}{b_{13}}\geq\frac{b_{24}}{a_{14}}$.
Thus, either $\frac{a_{14}}{b_{24}}>1$ or $\frac{a_{14}}{b_{13}}\geq1$. In any
case, we have that $P_{B,\gamma^{\prime}}[S_{3}^{(\beta)}]$ does not dominate
$P_{A,\beta}[S_{3}^{(\beta)}]$.
\end{itemize}

Then, $\beta\in\Gamma(B)$ and $P_{B,\beta}[S_{3}^{(\beta)}]\ $ dominates
$P_{A,\beta}[S_{3}^{(\beta)}].$ Similarly, since $\mathcal{E}(B)\subseteq
\mathcal{E}(A)$ and $\beta\in\Gamma(B)$, we get that $P_{A,\beta}%
[S_{3}^{(\beta)}]\ $ dominates $P_{B,\beta}[S_{3}^{(\beta)}].$ Thus,
$P_{A,\beta}[S_{3}^{(\beta)}]=P_{B,\beta}[S_{3}^{(\beta)}]$, implying $A=B.$
\end{proof}

\begin{lemma}
\label{lBp}Let $A,B\in\mathcal{PC}_{4}^{0}$ be as in (\ref{AB}), with
$\theta_{A}(A)=\theta_{B}(B)=\frac{1}{a_{14}}$. If $\mathcal{E}(A)=\mathcal{E}%
(B)$ and $\beta^{\prime}\in\Gamma(A)$, then $A=B.$
\end{lemma}

\begin{proof}
Suppose that $\beta^{\prime}\in\Gamma(A).$ Then, $\beta\notin\Gamma(B)$, as
otherwise, by Lemma \ref{lB}, $\ A=B,$ which contradicts the fact that
$\beta\notin\Gamma(A).$ We have $a_{13}<a_{24}$, as $\beta^{\prime}(A)<1.$
Since $\mathcal{E}(A)\subseteq\mathcal{E}(B),$ by Theorem \ref{lemsamerow},
there is $\rho\in\Gamma(B)$ such that $P_{B,\rho}[S_{1}^{(\beta^{\prime}%
)}]\leq P_{A,\beta^{\prime}}[S_{1}^{(\beta^{\prime})}].$ We will see that this
implies $\rho=\beta^{\prime}.$

\begin{itemize}
\item Since $\beta^{\prime}(A)<1,$ it follows that $P_{B,\alpha}[S_{1}%
^{(\beta^{\prime})}]$ $\ $ does not dominate $P_{A,\beta^{\prime}}%
[S_{1}^{(\beta^{\prime})}]$.

\item Suppose that $\gamma\in\Gamma(B)$. Since $\gamma(B)\geq\alpha(B),$ we
have $b_{13}b_{24}\geq1.$ Then, $P_{B,\gamma}[S_{1}^{(\beta^{\prime})}]$ does
not dominate $P_{A,\beta^{\prime}}[S_{1}^{(\beta^{\prime})}],$ since
$b_{24}\leq1$ and $b_{13}\leq\beta^{\prime}(A)<1$ do not occur simultaneously.

\item Suppose that $\gamma^{\prime}\in\Gamma(B)$. If $\frac{b_{24}}{a_{14}}%
>1$, then $\frac{a_{14}}{b_{13}}>1,$ as, from $\gamma^{\prime}(B)\geq
\alpha(B)$, we have $\frac{a_{14}}{b_{13}}\geq\frac{b_{24}}{a_{14}}.$ Thus,
either $\frac{a_{14}}{b_{24}}\geq1$ or $\frac{a_{14}}{b_{13}}>1.$ In any case,
we have that $P_{B,\gamma^{\prime}}[S_{1}^{(\beta^{\prime})}]$ does not
dominate $P_{A,\beta^{\prime}}[S_{1}^{(\beta^{\prime})}].$
\end{itemize}

Then $\beta^{\prime}\in\Gamma(B)$ and $P_{B,\beta^{\prime}}[S_{1}%
^{(\beta^{\prime})}]\ $ dominates $P_{A,\beta^{\prime}}[S_{1}^{(\beta^{\prime
})}].$ Similarly, since $\mathcal{E}(B)\subseteq\mathcal{E}(A)$ and
$\beta^{\prime}\in\Gamma(B)$, we get that $P_{A,\beta^{\prime}}[S_{1}%
^{(\beta^{\prime})}]\ $ dominates $P_{B,\beta^{\prime}}[S_{1}^{(\beta^{\prime
})}].$ Thus, $P_{A,\beta^{\prime}}[S_{1}^{(\beta^{\prime})}]=P_{B,\beta
^{\prime}}[S_{1}^{(\beta^{\prime})}]$, implying $A=B.$
\end{proof}

\begin{lemma}
\label{lG}Let $A,B\in\mathcal{PC}_{4}^{0}$ be as in (\ref{AB}), with
$\theta_{A}(A)=\theta_{B}(B)=\frac{1}{a_{14}}$. If $\mathcal{E}(A)=\mathcal{E}%
(B)$ and $\gamma\in\Gamma(A)$, then $A=B.$
\end{lemma}

\begin{proof}
Suppose that $\gamma\in\Gamma(A).$ If $\beta\in\Gamma(A)\cup\Gamma(B)$ or
$\beta^{\prime}\in\Gamma(A)\cup\Gamma(B)$, by Lemmas \ref{lB} and \ref{lBp},
$A=B$. Suppose that $\beta,\beta^{\prime}\notin\Gamma(A)\cup\Gamma(B)$. Since
$\mathcal{E}(A)\subseteq\mathcal{E}(B),$ by Theorem \ref{lemsamerow}, there is
$\rho\in\Gamma(B)$ such that $P_{B,\rho}[S_{2}^{(\gamma)}]\leq P_{A,\gamma
}[S_{2}^{(\gamma)}].$ We will see that this implies $\rho=\gamma.$

\begin{itemize}
\item Since $\gamma(A)<1,$ it follows that $P_{B,\alpha}[S_{2}^{(\gamma)}]\ $
does not dominate $P_{A,\gamma}[S_{2}^{(\gamma)}]$.

\item If $\gamma^{\prime}\in\Gamma(B),$ $P_{B,\gamma^{\prime}}[S_{2}%
^{(\gamma)}]\ $ does not dominate $P_{A,\gamma}[S_{2}^{(\gamma)}]$.
\end{itemize}

Then, $\gamma\in\Gamma(B)$ and $P_{B,\gamma}[S_{2}^{(\gamma)}]\ $ dominates
$P_{A,\gamma}[S_{2}^{(\gamma)}].$ Similarly, since $\mathcal{E}(B)\subseteq
\mathcal{E}(A)$ and $\gamma\in\Gamma(B)$, we get that $P_{A,\gamma}%
[S_{2}^{(\gamma)}]\ $ dominates $P_{B,\gamma}[S_{2}^{(\gamma)}].$ Thus,
$P_{A,\gamma}[S_{2}^{(\gamma)}]=P_{B,\gamma}[S_{2}^{(\gamma)}]$, implying
$A=B.$
\end{proof}

\begin{lemma}
\label{lGp}Let $A,B\in\mathcal{PC}_{4}^{0}$ be as in (\ref{AB}), with
$\theta_{A}(A)=\theta_{B}(B)=\frac{1}{a_{14}}$. If $\mathcal{E}(A)=\mathcal{E}%
(B)$ and $\gamma^{\prime}\in\Gamma(A)$, then $A=B.$
\end{lemma}

\begin{proof}
Suppose that $\gamma^{\prime}\in\Gamma(A).$ If $\beta\in\Gamma(A)\cup
\Gamma(B)$ or $\beta^{\prime}\in\Gamma(A)\cup\Gamma(B)$, by Lemmas \ref{lB}
and \ref{lBp}, $\ A=B.$ Suppose that $\beta,\beta^{\prime}\notin\Gamma
(A)\cup\Gamma(B)$. We have $\gamma\notin\Gamma(B)$, as otherwise, by Lemma
\ref{lG}, $\ A=B,$ which contradicts the fact that $\gamma\notin\Gamma(A).$
Since $\mathcal{E}(A)\subseteq\mathcal{E}(B),$ by Theorem \ref{lemsamerow},
there is $\rho\in\Gamma(B)$ such that $P_{B,\rho}[S_{4}^{(\gamma^{\prime}%
)}]\leq P_{A,\gamma^{\prime}}[S_{4}^{(\gamma^{\prime})}].$

Since $\gamma^{\prime}(A)<1$, we have $\frac{1}{a_{13}a_{24}}<\frac{1}{a_{14}%
}.$ Thus, $P_{B,\alpha}[S_{4}^{(\gamma^{\prime})}]\ $ does not dominate
$P_{A,\gamma^{\prime}}[S_{4}^{(\gamma^{\prime})}]$. Then, $\gamma^{\prime}%
\in\Gamma(B),$ and $P_{B,\gamma^{\prime}}[S_{4}^{(\gamma^{\prime})}]$
dominates $P_{A,\gamma^{\prime}}[S_{4}^{(\gamma^{\prime})}].$ Similarly, since
$\mathcal{E}(B)\subseteq\mathcal{E}(A)$ and $\gamma^{\prime}\in\Gamma(B)$, we
get that $P_{A,\gamma^{\prime}}[S_{4}^{(\gamma^{\prime})}]\ $ dominates
$P_{B,\gamma^{\prime}}[S_{4}^{(\gamma^{\prime})}].$ Thus, $P_{A,\gamma
^{\prime}}[S_{4}^{(\gamma^{\prime})}]=P_{B,\gamma^{\prime}}[S_{4}%
^{(\gamma^{\prime})}]$, implying $A=B.$
\end{proof}

\begin{theorem}
\label{thmain4}Let $A,B\in\mathcal{PC}_{4}.$ If $\mathcal{E}(A)=\mathcal{E}%
(B)$ then $A=B$.
\end{theorem}

\begin{proof}
If $A$ or $B$ is consistent, the result has already been mentioned. Suppose
that $A,B\in\mathcal{PC}_{4}^{0}$. By Theorem \ref{th4}, if $\mathcal{E}%
(A)=\mathcal{E}(B),$ there is a common cycle in $A$ and $B$ with the same
entries along it in both matrices $A$ and $B$ and with minimal cycle product.
So, by a monomial similarity, we may assume that both $A\ $and $B$ are as in
(\ref{AB}) with $\theta_{A}(A)=\theta_{B}(B)=\frac{1}{a_{14}}$. If
$\#\Gamma(A)=\#\Gamma(B)=1,$ then $\beta(A)=\gamma(A)=\beta(B)=\gamma(B)=1$
(for $\beta$ and $\gamma$ as in (\ref{abg})), implying $A=B.$ If
$\#\Gamma(A)>1$, the result follows from Lemmas \ref{lB} to \ref{lGp}.
\end{proof}

\section{Conclusions\label{scon}}

For a given $n$-by-$n$ inconsistent reciprocal matrix $A$, we have given a
union of at most $\frac{(n-1)!}{2}$ matrix intervals, depending on the entries
of $A$, such that a consistent matrix lies in it if and only if it is obtained
from an efficient vector for $A.$ Each interval is associated with a
Hamiltonian cycle product $<1$ in $A$. Based on this result, the partial order
on the alternatives dictated by the efficient vectors follows. The existence
of a unique order for the efficient vectors corresponding to each matrix
interval is considered in detail. We then focused on the question of when
$\mathcal{E}(A)=\mathcal{E}(B)$ implies $A=B$. We gave some general sufficient
conditions for the implication to hold, and have shown that it is correct when
both matrices $A$ and $B$ are simple perturbed consistent matrices (which
includes the case $n=3$) and when $n=4.$ It is not yet known if there are
counterexamples in any dimension greater than $4$.

\bigskip

There are no conflicts of interest.

\end{document}